\documentclass[11pt,twoside]{article}
\usepackage{amssymb}
\usepackage{amsmath}
\usepackage{amsthm}
\usepackage{color}
\usepackage{indentfirst}
\allowdisplaybreaks
\def\rr{{\mathbb R}}
\def\rn{{{\rr}^n}}
\def\ep{\varepsilon}
\def\phi{\varphi}

\def\bmoz{{{\rm BMO}(\rn)}}
\def\bloz{{{\rm BLO}(\rn)}}
\def\esup{\mathop\mathrm{\,esssup\,}}
\def\einf{{\mathop{\mathrm{\,essinf\,}}}}
\def\r{\right}
\def\lf{\left}

\newcommand{\abs}[1]{\left\vert#1\right\vert}

\newtheorem{thm}{Theorem}[section]
\newtheorem{lem}[thm]{Lemma}%[section]
\newtheorem{prop}[thm]{Proposition}%[section]
\newtheorem{cor}[thm]{Corollary}%[section]
\newtheorem{defn}[thm]{Definition}%[section]
\newtheorem{remark}[thm]{Remark}
\numberwithin{equation}{section}
\begin{document}
\arraycolsep=1pt
\title{\Large\bf  Boundedness of variation, oscillation and maximal differential transform on BMO space}
\author{Wenting Hu, Kai Wu, Dongyong Yang and Chao Zhang\thanks{Corresponding author} }
\pagestyle{myheadings}\markboth{Wenting Hu, Kai Wu, Dongyong Yang and Chao Zhang}
{ Boundedness of operators associated with approximate identities on $BMO$ space}
%\medskip
\date{}
\maketitle
\begin{center}
\begin{minipage}{13.5cm}\small
{\noindent  {\bf Abstract:}\ In this paper, we prove that the oscillation operator, variation operator and maximal differential transform associated with the approximate identities are bounded from $\bmoz$ to its subspace $\bloz$.}
\end{minipage}
\end{center}

\bigskip

\bigskip


\section{Introduction and main results}\label{intro}

Variational inequalities originated from the aim of improving the well-known Doob maximal
inequality. Relied on the work of L\'epingle \cite{L} for martingales, Bourgain \cite{B} obtained the variational estimates of Birkhoff ergodic averages and pointwise convergence results. These results induce a new research subject in harmonic analysis and
 ergodic theory, and since then, the study of boundedness of oscillation and variation operators  has been paid more and more attentions. Especially, the classical work of $\rho$-variation operators for singular integrals was given in \cite{CJRW}, in which the authors obtained the $L^{p}$ boundedness and weak type (1,1) boundedness for $\rho$-variation operators of truncated Hilbert transform for $\rho>2$, and then extended to higher dimensional cases in \cite{CJRWH}. For further studies, we refer readers to \cite{AJS, CMMTV, CMTT, Demir1, Demir2, JKRW, JOR, JR, JW, JSW, JQ} and the references therein. Especially, for the variation operators associated  with heat and Poisson semigroups, Jones et al. \cite{JR} and Crescimbeni et al. \cite{CMMTV} independently established the $L^{p}$-bounds and weak type (1,1) bounds by different approaches. Betancor et al. \cite{BFHR} obtained the boundedness of $\rho$-variation operators for classical Riesz transform on $\bmoz$.

 Recently, Liu \cite{LHH} generalized the results in \cite{CMMTV, JR} to the variation operators associated with approximate identities and obtained a variational characterization of Hardy spaces. Based on these results, one of our main purposes is to get the boundedness of variation operators associated with approximate identities on $\bmoz$, which extends the results in \cite{BFHR} for variation operators associated with the heat semigroup. In this paper, we will focus on the oscillation operators associated with approximate identities at first.  Due to the close connection between variation operators and oscillation operators, relevant results related to variation operators will also be established. We believe that these results in the classical setting are of independent interest. It should be noted that, although variation operators and oscillation operators are both $L^p$-bounded, they act differently on BMO functions. Before stating our main results, we first recall some relevant definitions and notations.
 Let $\{T_{t}\}_{t>0}$ be a family of operators. For $\rho>2$, we define the $\rho$-variation for $\{T_{t}\}_{t>0}$ by
 \begin{equation*}
 	 \mathcal{V}_{\rho}(\{T_{t}\}_{t>0},f)(x):=\sup\limits_{\ep_{i}\searrow 0}\lf ( \sum\limits_{i=1}^{\infty}|{T_{\ep_{i}}f(x)-T_{\ep_{i+1}}f(x)}|^{\rho}\r)
 	^{\frac{1}{\rho}},
 \end{equation*}
 where the supremum is taken over all the sequences of positive number $\{\ep_{i}\}_{i=1}^{\infty}$ decreasing to zero.  It is worth noting that, in order to obtain
 the boundedness of variations, the assumption $\rho>2$ is necessary; see, for example, \cite[Remark 1.7]{CJRW} or \cite{AJS,JQ}. Let $\mathbb{N}:=\{1,2,\cdots\}$ and $\{t_{i}\}_{i\in\mathbb{N}}\subset (0,\infty)$ be a fixed sequence decreasing to zero. Then the oscillation operators of the family $\{T_{t}\}_{t>0}$ is defined by
\begin{equation*}
	\mathcal{O}_{\{t_{i}\}_{i\in\mathbb{N}}}\left(\{T_{t}\}_{t>0}, f\right)(x):=
	\left ( \sum\limits_{i=1}^{\infty}\sup\limits_{t_{i+1}\leq \varepsilon_{i+1}<\varepsilon_{i}\leq t_{i} }|{T_{\ep_{i}}f(x)-T_{\ep_{i+1}}f(x)}|^{2}
	 \right )^{\frac{1}{2}}.
\end{equation*}

We also recall the definitions of the $\rm{BMO}$ and $\rm{BLO}$ function spaces in the following.
\begin{defn}
A function $f\in L^1_{\rm loc}(\rn)$ belongs to
the {\it space} $\bmoz,$ if
\begin{equation*}
\|f\|_{\bmoz}:=\sup_{B}\frac1{|B|}\int_{B}|f(x)-f_{B}|dx<\infty,
\end{equation*}
where  the supremum is taken over all balls $B$ in $\rn$, $\left| B\right| $ denotes the Lebesgue measure of $B$ and
\begin{equation*}
f_{B}:=\frac 1{|B|}\int_{B}f(x)dx.
\end{equation*}
\end{defn}
\begin{defn}\label{BLO def}
	A function $f\in L^1_{\rm loc}(\rn)$ belongs to
	the {\it space} $\bloz$ if
	\begin{equation*}
		\|f\|_{\bloz}:=\sup_{B}\frac1{|B|}\int_{B}\lf [f(x)-\mathop{\einf}\limits_{y\in B}f(y)\r ] dx <\infty,
	\end{equation*}
	where  the supremum is taken over all balls $B$ in $\rn$.
\end{defn}
\par Let $\phi\in \mathcal{S}(\rn)$ with $\int_{\rn}\phi(x)dx=1$, where $\mathcal{S}(\rn)$ is the space of Schwartz functions (see \eqref{S space} below). We consider the approximate identities $\{\phi_{t}\}_{t>0}$, where $\phi_{t}(x):=t^{-n}\phi(x/t)$. Let $\{\Phi_{t}\}_{t>0}$ be the family of convolution operators generated by $\{\phi_{t}\}_{t>0}$, that is,
\begin{equation}\label{convo}
	\Phi_{t}f(x):=\phi_{t}\ast f(x).
\end{equation}
Now we formulate our first main result of this paper as follows:
\begin{thm}\label{main O}
Let $\phi\in \mathcal{S}(\rn)$ satisfying $\int_{\rn}\phi(x)dx=1$, $\{t_{i}\}_{i\in\mathbb{N}}\subset (0,\infty)$ be a fixed sequence decreasing to zero and $\{\Phi_{t}\}_{t>0}$ be as in \eqref{convo}. Then for any $f\in \bmoz$, we have that $\mathcal{O}_{\{t_{i}\}_{i\in\mathbb{N}}}\left(\{\Phi_{t}\}_{t>0}, f \right)(x)<\infty,a.e.\ x\in\rn,$ and there exists $C>0$  depending on $n$ and $\phi$ such that
$$\|\mathcal{O}_{\{t_{i}\}_{i\in\mathbb{N}}}\left(\{\Phi_{t}\}_{t>0}, f \right)\|_{\bloz}\leq C\|f\|_{\bmoz}.$$
\end{thm}

Since $\bloz\subset\bmoz$, as a corollary of Theorem \ref{main O}, we can conclude the boundedness of $\mathcal{O}_{\{t_{i}\}_{i\in\mathbb{N}}}\left(\{\Phi_{t}\}_{t>0}, \cdot \right)$ on $\bmoz$ in the following.

\begin{cor}\label{cor O}
 Let $\phi, \{\Phi_{t}\}_{t>0}$ and $\{t_{i}\}_{i\in\mathbb{N}}$ be as in Theorem \ref{main O}. Then the oscillation operator $\mathcal{O}_{\{t_{i}\}_{i\in\mathbb{N}}}\left(\{\Phi_{t}\}_{t>0}, \cdot \right)$ is bounded on $\bmoz$.
\end{cor}

Moreover, we obtain similar but not identical results with regard to variation operators as follows.
\begin{thm}\label{main V}
	Let $\phi\in \mathcal{S}(\rn)$ satisfying $\int_{\rn}\phi(x)dx=1$, $\{\Phi_{t}\}_{t>0}$ be be as in \eqref{convo}. Then for $f\in \bmoz$,  $ \mathcal{V}_{\rho}\left(\{\Phi_{t}\}_{t>0},f\right)$ is either  infinite everywhere or finite almost everywhere. And in the latter case, there exists $C>0$  depending on $n$ and $\phi$ such that
		$$\|\mathcal{V}_{\rho}\left( \{\Phi_{t}\}_{t>0},f\right)\|_{\bloz}\leq C\|f\|_{\bmoz}.$$
\end{thm}

Similar to Corollary \ref{cor O}, we have the following corollary.
 \begin{cor}\label{cor V}
 	Let $\phi, \rho$ and $\{\Phi_{t}\}_{t>0}$ be as in Theorem \ref{main V}. Then there exists $C>0$  depending on $n$ and $\phi$ such that
 	$$\|\mathcal{V}_{\rho}\left( \{\Phi_{t}\}_{t>0},f\right)\|_{\bmoz}\leq C\|f\|_{\bmoz},$$
 	for any $f\in \bmoz$ with $\mathcal{V}_{\rho}\left( \{\Phi_{t}\}_{t>0},f\right)(x_0)<\infty$ for some $x_{0}\in\rn$.
 \end{cor}
Let $\mathbb{Z}_{<}^{2}:=\{(N_{1},N_{2})\in \mathbb{Z}^{2}: N_{1}<N_{2}\}$. We now consider the maximal differential transform associated with the family $\{\Phi_{t}\}_{t>0}$ defined by
$$\mathcal{S}_{\boldsymbol{a,v};*}\left(\{\Phi_{t}\}_{t>0},f \right)(x):=\sup\limits_{(N_{1},N_{2})\in \mathbb{Z}_{<}^{2}  } \left| \sum\limits_{i=N_{1}}^{N_{2}}v_{i}\left(\Phi_{a_{i+1}}f(x)-\Phi_{a_{i}}f(x) \right) \right|,$$
where $\boldsymbol{a}:=\{a_{i}\}_{i\in \mathbb{Z}}$ is a $\delta$-lacunary sequence of positive numbers, that is, there exists a constant $\delta>1$ such that $a_{i+1}/ a_{i}\geq \delta$ for ecah $i\in \mathbb{Z}$, and $\boldsymbol{v}:=\{v_{i}\}_{i\in \mathbb{Z}}$ is a bounded sequence of complex numbers. We remark that, by this way to analyze convergence of sequence was considered by Jones and Rosenblatt \cite{JRD} for ergodic average, and latter by Bernardis et al. \cite{BLMTTT}   and Zhang et al. \cite{CMT,CT}  for differential transform. Next, we  will consider the boundedness of $\mathcal{S}_{\boldsymbol{a,v};*}\left(\{\Phi_{t}\}_{t>0}, \cdot \right)$ on $L^{p}(\rn)$ space with $1<p<\infty$.
\begin{thm}\label{S Lp}
	Let $\boldsymbol{a}$ be a $\delta$-lacunary sequence of positive numbers, $\boldsymbol{v}$ is a bounded sequence of complex numbers and $\phi\in \mathcal{S}(\rn)$ with $ \int_{\rn}\phi(x)dx=1$. Then for any $1<p<\infty$, there exists  a constant $C>0$  depending on $n, p, \left\| \boldsymbol{v}
	\right\|_{l^{\infty}(\mathbb{Z})},\ \delta$ and $\phi$ such that  $$\left\| \mathcal{S}_{\boldsymbol{a,v};*}\left(\{\Phi_{t}\}_{t>0}, f \right)\right\| _{L^p(\rn)}\leq C\|f\|
	_{L^p(\rn)},$$
for all functions $f\in L^{p}(\rn)$.
	\end{thm}

\begin{remark}\rm
\begin{itemize}
  \item [{\rm (i)}] When $\phi(x):=\frac{c_n}{(|x|^2+1)^{\frac{n+1}{2}}}$ with $c_n:=\frac{\Gamma({\frac{n+1}{2}})} {\pi^{\frac{n+1}{2}}}$, then $\phi_t(x)=c_n\frac{t}{ (t^2+|x|^2)^{\frac{n+1}{2}}}$ is just the Poisson kernel, and $\{\Phi_t\}_{t>0}$ is the classical Poisson semigroup $\{e^{-t\sqrt{-\Delta}}\}_{t>0}$ generated by $-\Delta$. In this case, Theorem \ref{S Lp} can be gotten by using the Bochner subordination formula with the results which were proved in \rm{\cite{CT}}.
  \item [{\rm (ii)}] From Theorem \ref{S Lp} and a standard argument, we can also obtain the weak-type (1,1) boundedness of $\mathcal{S}_{\boldsymbol{a,v};*}$.
\end{itemize}
\end{remark}

Based on the result in Theorem \ref{S Lp}, we can prove the boundedness of $\mathcal{S}_{\boldsymbol{a,v};*}\left(\{\Phi_{t}\}_{t>0}, \cdot \right)$ from $\bmoz$ to $\bloz$. Following some idea in \cite{AB}, we need assume that the $\delta$-lacunary sequence $\boldsymbol{a}:=\{a_{i}\}_{i\in \mathbb{Z}}$ satisfies the following condition:
\begin{equation}\label{add}
	1<\delta\leq \frac{a_{i+1}}{a_{i}}\leq \delta^{2},\quad \quad  \forall\ i\in \mathbb{Z}.
\end{equation}

\begin{thm}\label{main 3}
	Let $\boldsymbol{a}$ be a $\delta$-lacunary sequence of positive numbers satisfying $\eqref{add}$, $\boldsymbol{v}$ be a bounded sequence of complex numbers and $\phi\in \mathcal{S}(\rn)$ with $\int_{\rn}\phi(x)dx=1$. Then for $f\in \bmoz$, $\mathcal{S}_{\boldsymbol{a,v};*}\left(\{\Phi_{t}\}_{t>0}, f \right)$ is either  infinite everywhere or finite almost everywhere. And in the latter case, there exists  a constant $C>0$  depending on $n, \left\| \boldsymbol{v}\right\|_{l^{\infty}(\mathbb{Z})}, \delta$ and $\phi$ such that
	 $$\left\| \mathcal{S}_{\boldsymbol{a,v};*}\left(\{\Phi_{t}\}_{t>0}, f \right)\right\| _{\bloz}\leq C\|f\|
	_{\bmoz}.$$
\end{thm}

From the results in Theorems \ref{main O}, \ref{main V} and \ref{main 3}, we should note that, when $f\in \rm{BMO(\rn)},$ $\mathcal{O}_{\{t_{i}\}_{i\in\mathbb{N}}}\left(\{\Phi_{t}\}_{t>0}, f \right)(x)<\infty,a.e.\ x\in\rn,$ but for the variation operator $ \mathcal{V}_{\rho}\left(\{\Phi_{t}\}_{t>0},f\right)$ and maximal differential transform $\mathcal{S}_{\boldsymbol{a,v};*}\left(\{\Phi_{t}\}_{t>0}, f \right)$, they are either  infinite everywhere or finite almost everywhere. This fact is the difference among these three operators when they act on the endpoint spaces.

This paper is organized as follows. Section \ref{pre} is devoted to give some known estimates on $\{\phi_{t}\}_{t>0}$, as well as some useful properties of the functions in $\bmoz$. To better study the behavior of the maximal differential transform, we also analyze its ``partial sums" in this section. The proof of Theorems \ref{main O}, \ref{main V}, \ref{S Lp} and  \ref{main 3} are presented in Sections \ref{sec3}-\ref{sec4}, respectively. We obtain these results by following some idea from \cite{HBLWY}, see also \cite{LiY, LY, YY, YYZ, Z}.  We also remark that, to our best knowledge, the boundedness of variation, oscillation operators and maximal differential transforms from $\bmoz$ to $\bloz$ are first given in this paper.
\par
Throughout the paper, $C$ and $c$ are positive constants, which may vary from line to line. For any $k\in \mathbb{R}_+$ and ball $B:=B(x_{0}, r)$, $kB:=B(x_{0},kr)$.

\section{Preliminaries}\label{pre}

At first, we present some preliminary results on $\{\phi_{t}\}_{t>0}$, which are widely used throughout our paper. For convenience, we recall the definition of the Schwartz functions. A smooth real-valued function $\phi$ on $\rn$ is called a Schwartz function, if for every pair of multi-indices $\alpha$ and $\beta$,
\begin{equation}\label{S space}\rho_{\alpha,\beta}(\phi):= \sup\limits_{x\in \rn}\left|x^{\alpha}\partial^{\beta}\phi(x)\right|<\infty.
\end{equation}
\begin{lem}
For $\phi\in \mathcal{S}(\rn)$, and $\phi_{t}(x):=t^{-n}\phi(x/t),\ t>0$, there exists a constant $C>0$ depending on $n$ and $\phi$ such that
\begin{itemize}
	 \item [{\rm (i)}]for any $x\in \rn$,
	\begin{equation}\label{one partial}
		\left| \frac{\partial}{\partial t} \left( \phi_{t}(x)\right) \right| \leq
		\frac{ C}{(t+\left|x \right|)^{n+1}};
	\end{equation}
	\item [{\rm(ii)}] for any $x\in B(x_{0},r)$ with $x_{0} \in \rn$ and $r>0$, $y\in \left( B(x_{0},2r)\right)^{c} $,
	\begin{align}
		\left| \frac{\partial}{\partial t}\left( \phi_{t}(x-y)\right) \right| &\leq  \frac{C}{\left| y-x_{0}\right|^{n+1} };\label{two outer partial}	
	\end{align}
	\item [{\rm(iii)}]for any $x, y\in B(x_{0},r)$ with $x_{0} \in \rn$ and $r>0$, $z\in B(x_{0},2r)$,
	\begin{align}
		\left| \frac{\partial}{\partial t}\left( \phi_{t}(x-z)-\phi_{t}(y-z)\right) \right| &\leq C\frac{\left|x-y \right| }{t^{n+2}} \label{three inner partial};
	\end{align}
	\item [{\rm (iv)}]for any $x, y\in B(x_{0},r)$ with $x_{0} \in \rn$ and $r>0$, $z\in \left( B(x_{0},2r)\right)^{c}$,
	\begin{align}
		\left| \frac{\partial}{\partial t}\left( \phi_{t}(x-z)-\phi_{t}(y-z)\right) \right| &\leq C\frac{\left|x-y \right| }{t^{\frac{3}{2}}\left|z-x_{0} \right|^{n+\frac1{2}}}\label{three outer partial}.
	\end{align}
\end{itemize}
\end{lem}

\begin{proof}
	For simplicity, we only give the proof of \eqref{two outer partial}, \eqref{three inner partial} and \eqref{three outer partial}. By a direct calculation, we show \eqref{two outer partial} as follows.
\begin{align*}
\left| \frac{\partial}{\partial t}\left( \phi_{t}(x-y)\right) \right| &=\left|\frac{n\phi(\frac{x-y}{t})}{t^{n+1}}+
\frac{\left(x-y \right)\cdot\nabla \phi\left( \frac{x-y}{t}\right)}{t^{n+2}} \right|\leq \frac{C}{\left| y-x_{0}\right|^{n+1} }.		
\end{align*}

For \eqref{three inner partial}, applying the mean value theorem, there exist $\xi$ and $\eta$ on the segment $\overline{(x-z)(y-z)}$ such that
\begin{align*}
	&\left| \frac{\partial}{\partial t}\left( \phi_{t}(x-z)-\phi_{t}(y-z)\right) \right|\\&\quad =\Bigg|\left[ \frac{n\phi(\frac{x-z}{t})}{t^{n+1}}+
	\frac{\left(x-z \right)\cdot\nabla\phi\left( \frac{x-z}{t}\right)  }{t^{n+2}}\right]-\left[ \frac{n\phi\left( \frac{y-z}{t}\right) }{t^{n+1}}+\frac{\left(y-z \right)\cdot\nabla\phi\left( \frac{y-z}{t}\right)  }{t^{n+2}}\right]\Bigg|\\
	&\quad\leq C\frac{\left|x-y \right| }{t^{n+2}}\left[   \left|\nabla\phi\left( \frac{\xi}{t}\right)  \right|+ \left|\nabla\left(\frac{\eta}{t} \cdot\nabla\phi\left( \frac{\eta}{t}\right) \right) \right|\right]\leq C\frac{\left|x-y \right| }{t^{n+2}}.
\end{align*}

For \eqref{three outer partial}, as previous proof, there exist $\xi$ and $\eta$ on the segment $\overline{(x-z)(y-z)}$ such that
\begin{align*}
&\left| \frac{\partial}{\partial t}\left( \phi_{t}(x-z)-\phi_{t}(y-z)\right) \right|\\&\quad	\leq C \frac{\left|x-y \right| }{t^{n+2}}\left[   \left|\nabla\phi\left( \frac{\xi}{t}\right)\right|+ C \left|\nabla\left( \frac{\eta}{t}\cdot\nabla\phi\left( \frac{\eta}{t}\right) \right) \right|\right]  \\
&\quad\leq C\frac{\left|x-y \right| }{t^{\frac{3}{2}}\left| \xi\right|^{n+\frac1{2}} }\left| \frac{ \xi}{t}\right|^{n+\frac1{2}}\left|\nabla\phi\left( \frac{\xi}{t}\right)  \right|+C\frac{\left|x-y \right| }{t^{\frac{3}{2}}\left| \eta\right|^{n+\frac1{2}} }\left| \frac{ \eta}{t}\right|^{n+\frac1{2}}\left|\nabla\left(\frac{\eta}{t}\cdot \nabla\phi\left( \frac{\eta}{t}\right) \right) \right|\\
&\quad\leq C\frac{\left|x-y \right| }{t^{\frac{3}{2}}\left|z-x_{0} \right|^{n+\frac1{2}}},	
\end{align*}
where in the last inequality we have used that $\left| \xi\right|,\left| \eta\right| \sim \left|z-x_{0} \right|.$
\end{proof}
Now we recall some well known properties for the functions in $\bmoz$, whose proofs can be found for example in \cite{GTM}.

\begin{lem}\label{BMO pro}
For $f\in\bmoz$, any ball $B\subset\rn$ and positive integer $m$, there exists $C>0$ depending only on $n $ such that the following statements hold:
\begin{itemize}
	\item [{\rm(i)}]
	\begin{equation*}
		\left|f_{B}-f_{2^{m}B} \right|\leq C m\|f\|_{\bmoz};
	\end{equation*}
	\item [{\rm(ii)}]
	\begin{equation*}
		\sup_{B}\left( \frac1{\left|B \right| }\int_{B} \left| f(x)-f_{B}\right|^{2}dx\right)^\frac1{2} \leq C\|f\|_{\bmoz},
	\end{equation*}
where  the supremum is taken over all balls $B$ in $\rn$;
    \item [{\rm(iii)}]
    \begin{equation*}
    	\frac1{\left|2^{m}B \right| }\int_{2^{m}B} \left| f(x)-f_{B}\right| dx\leq C m\|f\|_{\bmoz}.
    \end{equation*}
\end{itemize}
\end{lem}
The $L^{p}$ boundedness properties for the variation operator $\mathcal{V}_{\rho}(\{\Phi_{t}\}_{t>0},\cdot)$ and the oscillation operator  $ \mathcal{O}_{\{t_{i}\}_{i\in\mathbb{N}}}\left(\{\Phi_{t}\}_{t>0}, \cdot \right)$, were studied in \cite{LHH}. We provide here the precise statement as follows.
\begin{lem}\label{O and V Lp}
	Let $\phi, \rho, \{t_{i}\}_{i\in\mathbb{N}} $ and $\{\Phi_{t}\}_{t>0}$ be as in Theorems \ref{main O} and \ref{main V}. Then the variation operator $\mathcal{V}_{\rho}(\{\Phi_{t}\}_{t>0},\cdot)$ and  the oscillation operator $ \mathcal{O}_{\{t_{i}\}_{i\in\mathbb{N}}}\left(\{\Phi_{t}\}_{t>0}, \cdot \right)$ are both bounded from $L^{p}(\rn)$ into $L^{p}(\rn)$, for every $1<p<\infty$.
\end{lem}
In order to better understand the behavior of the maximal differential transform, we shall analyze its ``partial sums" defined as follows. For each $N:=(N_{1},N_{2})\in \mathbb{Z}_{<}^{2}$, $f\in L^{p}(\rn), 1\leq p\leq \infty$, we define the partial sum operators
\begin{equation}\label{partial sum}
	\mathcal{S}_{\boldsymbol{a,v};N}\left(\{\Phi_{t}\}_{t>0}, f \right)(x):=\sum\limits_{i=N_{1}}^{N_{2}}v_{i}\left(\Phi_{a_{i+1}}f(x)-\Phi_{a_{i}}f(x) \right), x\in \rn.
\end{equation}
Then the maximal differential transform can be expressed as
$$\mathcal{S}_{\boldsymbol{a,v};*}\left(\{\Phi_{t}\}_{t>0}, f \right)(x)=\sup\limits_{N=(N_{1},N_{2})\in \mathbb{Z}_{<}^{2}}\left|\mathcal{S}_{\boldsymbol{a,v};N}\left(\{\Phi_{t}\}_{t>0}, f \right)(x) \right| .$$
For $f\in L^{p}(\rn), 1\leq p\leq \infty$, then
$$\mathcal{S}_{\boldsymbol{a,v};N}\left(\{\Phi_{t}\}_{t>0}, f \right)(x)=\int_{\rn}K_{\boldsymbol{a,v};N}(y)f(x-y)dy$$
with the kernel
\begin{equation}\label{Kernel}
K_{\boldsymbol{a,v};N}(y):=\sum\limits_{i=N_{1}}^{N_{2}}v_{i}\left(\phi_{a_{i+1}}(y)-\phi_{a_{i}}(y) \right).
\end{equation}
The following lemma contains the size description  of the kernel and the smoothness estimates that are required in the Calder\'{o}n-Zygmund theory.
\begin{lem}\label{ S kernel es }
There exists constant $C>0$ depending on $n, \left\| \boldsymbol{v}
\right\|_{l^{\infty}(\mathbb{Z})} $ and $\phi$ (not on $N$) such that, for any $y\neq 0$,
\begin{itemize}
	\item [{\rm(i)}] $\left| K_{\boldsymbol{a,v};N}(y)\right|\leq \frac{C}{\left| y\right|^{n} }$,
	\item [{\rm(ii)}]$\left| \nabla K_{\boldsymbol{a,v};N}(y)\right|\leq  \frac{C}{\left| y\right|^{n+1} }$.
\end{itemize}
\end{lem}
\begin{proof}
For (i), by the fact that $\boldsymbol{v}$ is bounded, as well as Newton-Leibniz formula, we obtain
\begin{align*}
	\left| K_{\boldsymbol{a,v};N}(y)\right| &\leq \sum\limits_{i=N_{1}}^{N_{2}}\left| v_{i}\right| \left| \phi_{a_{i+1}}(y)-\phi_{a_{i}}(y)\right| \\
	&\leq C \sum\limits_{i=-\infty}^{\infty}\left|\int_{a_{i}}^{a_{i+1}}\frac{\partial}{\partial t}\left(\phi_{t}(y) \right) dt\right|\\
	&\leq C \int_{0}^{\infty} \left|\frac{n\phi(\frac{y}{t})}{t^{n+1}}+
	\frac{y\cdot\nabla\phi\left( \frac{y}{t}\right)  }{t^{n+2}} \right|dt\\
	&\leq  C\left( \int_{0}^{ \left|y\right| }\left|\frac{n\phi(\frac{y}{t})}{t^{n+\frac1{2}}t^{\frac1{2}}}+\frac{y\cdot\nabla\phi\left( \frac{y}{t}\right)}{t^{n+\frac{3}{2}}t^{\frac1{2}}}\right|dt+\int_{\left|y\right|}^{ \infty }\left|\frac{n\phi(\frac{y}{t})}{t^{n-\frac1{2}}t^{\frac{3}{2}}}+\frac{y\cdot\nabla\phi\left( \frac{y}{t}\right)}{t^{n+\frac1{2}}t^{\frac{3}{2}}}\right|dt \right)\\
	&\leq \int_{0}^{ \left|y\right| }\frac{C}{t^{\frac1{2}}\left| y\right|^{n+\frac1{2}} }dt+ \int_{ \left|y\right|}^{\infty }\frac{C}{t^{\frac{3}{2}}\left| y\right|^{n-\frac1{2}} }dt\leq \frac{C}{\left|y \right|^{n} }.
\end{align*}
For (ii), according to \eqref{S space}, we have $\frac{\partial\phi}{\partial y_{i}}\in\mathcal{S}(\rn), i=1,2,\cdots, n. $ Then by a similar discussion to (i), we deduce
\begin{align*}
\left| \nabla K_{\boldsymbol{a,v};N}(y)\right|&\leq \sum\limits_{i=1}^{n}\left|\frac{\partial}{\partial y_{i}} K_{\boldsymbol{a,v};N}(y)\right|\\
&\leq C\sum\limits_{i=1}^{n}
 \int_{0}^{\infty}\left|\frac{\partial}{\partial t}\frac{\partial}{\partial y_{i}} \phi_{t}(y)\right|dt\\
 &= C\sum\limits_{i=1}^{n}\int_{0}^{\infty}\left|\frac{\partial}{\partial t}\left( \frac{\frac{\partial\phi}{\partial y_{i}}(\frac{y}{t})}{t^{n+1}}\right)  \right| dt\leq \frac{C}{\left|y \right|^{n+1} }.
\end{align*}
\end{proof}

\section{Boundedness of the oscillation operators and variation operators}\label{sec3}
In this section, we will establish the boundedness of $ \mathcal{O}_{\{t_{i}\}_{i\in\mathbb{N}}}\left(\{\Phi_{t}\}_{t>0}, \cdot \right)$ and $\mathcal{V}_{\rho}(\{\Phi_{t}\}_{t>0},\cdot)$ from $\bmoz$ to $\bloz$.
\begin{proof}[\bf Proof of Theorem \ref{main O}]
For $f\in \bmoz$, we first prove that,  $\mathcal{O}_{\{t_{i}\}_{i\in\mathbb{N}}}\left(\{\Phi_{t}\}_{t>0}, f \right)(x)<\infty,a.e.\ x\in\rn.$ Since $\{t_{i}\}_{i\in\mathbb{N}}\subset (0,\infty)$ is a fixed sequence decreasing to zero, there exists $r>0$ such that $t_{1}<r$. Now we take a ball $B:=B(x_0, r)\subset \rn$ with $x_{0}\in \rn$. We claim that, there exists $C>0$  depending on $n$ and $\phi$ such that
\begin{equation}\label{O a.e}
G:= \frac1{|B|}\int_{B}\mathcal{O}_{\{t_{i}\}_{i\in\mathbb{N}}}\left(\{\Phi_{t}\}_{t>0}, f \right)(x)dx \leq C \|f\|_{\bmoz}.
\end{equation}

By writing
$$f=(f-f_{B})\chi_{2B}+(f-f_{B})\chi_{(2B)^{c}}+f_{B}=:f_{1}+f_{2}+f_{3},$$
then we get
$$	G\leq \sum\limits_{j=1}^{3}\frac1{|B|}\int_{B}\mathcal{O}_{\{t_{i}\}_{i\in\mathbb{N}}}\left(\{\Phi_{t}\}_{t>0}, f_{j} \right)(x) dx=: G_{1}+G_{2}+G_{3}.$$

Noting that $\int_{\rn}\phi_{t}(x)dx=1$, it leads to that for any $t>0, x\in \rn$, $\Phi_{t}(f_{3})(x)=f_{B}.$ Furthermore, we get $G_{3}=0$.

For $G_{1}$, by  H\"older's inequality, Lemma \ref{BMO pro} (i), (ii) and $L^{2}$ boundedness of the oscillation operators, we get
\begin{align}
	G_{1}&\leq \left\lbrace\frac1{|B|}\int_{B}\lf [\mathcal{O}_{\{t_{i}\}_{i\in\mathbb{N}}}\left(\{\Phi_{t}\}_{t>0}, f_{1} \right)(x) \r ]^{2}dx \right\rbrace   ^{\frac1{2}}\notag\\
	&\leq C \left[  \frac1{|B|}\int_{2B}\abs{f(x)-f_{B}}^{2}dx\right] ^{\frac{1}{2}}\notag\\
	&\leq C \left[  \frac1{|2B|}\int_{2B}\abs{f(x)-f_{2B}}^{2}dx\right] ^{\frac{1}{2}}+C \left| f_{B}-f_{2B}\right|\notag \\
	&\leq C\|f\|_{\bmoz}.\label{G1}
\end{align}

For $G_{2}$, by the definition of oscillation, Newton-Leibniz formula, Fubini's theorem and \eqref{two outer partial} for any $x\in B$,
\begin{align*}
	 \mathcal{O}_{\{t_{i}\}_{i\in\mathbb{N}}}\left(\{\Phi_{t}\}_{t>0}, f_{2} \right)(x)&\leq \sum\limits_{i=1}^{\infty}\sup\limits_{t_{i+1}\leq \varepsilon_{i+1}<\varepsilon_{i}\leq t_{i} }\left| \int_{\varepsilon_{i+1}}^{\varepsilon_{i}}\frac{\partial}{\partial t}\Phi_{t}f_{2}(x)dt\right|\\
	&\leq \int_{\rn}\left| f_{2}(y)\right|\int_{0}^{r}\left|\frac{\partial}{\partial t} \phi_{t}(x-y)\right|dtdy\\
	&\leq C\int_{(2B)^{c}}\frac{\left|f(y)-f_{B} \right|r}{\left| y-x_{0}\right|^{n+1}}dy.
\end{align*}
Then by Lemma \ref{BMO pro} (iii), we have
\begin{align}
	G_{2}&\leq C\int_{(2B)^{c}}\frac{\left|f(y)-f_{B} \right|r}{\left| y-x_{0}\right|^{n+1}}dy\notag\\
	&\leq C\sum\limits_{k=1}^{\infty}\int_{2^{k+1}B\setminus 2^{k}B }\left|f(y)-f_{B} \right|\frac{r}{(2^kr)^{n+1}}dy\notag\\
	&\leq C \sum\limits_{k=1}^{\infty}\frac1{2^{k}}\frac1{\left|2^{k+1}B \right| }\int_{2^{k+1}B}\left| f(y)-f_{B}\right|dy\notag\\
	&\leq C\sum\limits_{k=1}^{\infty}\frac{k+1}{2^{k}}\|f\|_{\bmoz}\leq C \|f\|_{\bmoz}.\label{G2}
\end{align}

Summing the estimates of $G_{i}, i=1, 2, 3$, \eqref{O a.e} holds true. Then we obtain $$\mathcal{O}_{\{t_{i}\}_{i\in\mathbb{N}}}\left(\{\Phi_{t}\}_{t>0}, f \right)(x)<\infty,a.e.\ x\in B(x_0, r).$$
Let $r$ tend to infinity, we get
\begin{equation}\label{Oae}
	 \mathcal{O}_{\{t_{i}\}_{i\in\mathbb{N}}}\left(\{\Phi_{t}\}_{t>0}, f \right)(x)<\infty,a.e.\ x\in \rn.
\end{equation}

Based on the discussion above, in order to complete the proof, we only need to show that there exists $C>0$ depending on $n$ and $\phi$ such that for any ball $B:=B(x_{0},r)\subset \rn$,
\begin{equation}\label{main result O}
\frac1{|B|}\int_{B}\lf [ \mathcal{O}_{\{t_{i}\}_{i\in\mathbb{N}}}\left(\{\Phi_{t}\}_{t>0}, f \right)(x)-\mathop{\einf}\limits_{y\in B} \mathcal{O}_{\{t_{i}\}_{i\in\mathbb{N}}}\left(\{\Phi_{t}\}_{t>0},f \right)(y)\r ]dx\leq C \|f\|_{\bmoz}.
\end{equation}
Since $\{t_{i}\}_{i\in\mathbb{N}}\subset (0,\infty)$ is a fixed sequence decreasing to zero, we prove \eqref{main result O} by discussing the following two cases.

  Case 1: any $i\in \mathbb{N}$, $t_{i}<r$. In this case, from \eqref{G1} and \eqref{G2}, it is obvious that
   \begin{align*}
  	&\frac1{|B|}\int_{B}\lf [ \mathcal{O}_{\{t_{i}\}_{i\in\mathbb{N}}}\left(\{\Phi_{t}\}_{t>0}, f \right)(x)-\mathop{\einf}\limits_{y\in B} \mathcal{O}_{\{t_{i}\}_{i\in\mathbb{N}}}\left(\{\Phi_{t}\}_{t>0},f \right)(y)\r ]dx\\
  	&\quad \leq \frac1{|B|}\int_{B}\lf [\mathcal{O}_{\{t_{i}\}_{i\in\mathbb{N}}}\left(\{\Phi_{t}\}_{t>0}, f \right)(x) \r ]dx\leq C \|f\|_{\bmoz}.
  \end{align*}

Case 2: there exists an integer $i_{0}$ such that
$t_{i_{0}}<r\leq t_{i_{0}-1}$. Without loss of generality, we may further assume $i_{0}>3$.
In order to estimate
$$F:=\int_{B}\left[  \mathcal{O}_{\{t_{i}\}_{i\in\mathbb{N}}}\left(\{\Phi_{t}\}_{t>0}, f \right)(x)-\mathop{\einf}\limits_{y\in B} \mathcal{O}_{\{t_{i}\}_{i\in\mathbb{N}}}\left(\{\Phi_{t}\}_{t>0},f \right)(y)\right] dx,$$
we consider the following three subsets of $B$.
\begin{align*}
B_{1}:&= \left\{y\in B:\mathop{\sup}\limits_{t_{i_{0}}\leq \ep_{i_{0}}< \ep_{i_{0}-1}\leq t_{i_{0}-1} } \left|\Phi_{\ep_{i_{0}-1}}f(y)-\Phi_{\ep_{i_{0}}}f(y) \right|\right.
\\
&\left.\quad\quad\quad\quad\quad\quad\quad\quad=\mathop{\sup}\limits_{t_{i_{0}}\leq \ep_{i_{0}}< \ep_{i_{0}-1}\leq r } \left|\Phi_{\ep_{i_{0}-1}}f(y)-\Phi_{\ep_{i_{0}}}f(y) \right|\right\},	
\end{align*}
\begin{align*}
	B_{2}:&= \left\{y\in B:\mathop{\sup}\limits_{t_{i_{0}}\leq \ep_{i_{0}}< \ep_{i_{0}-1}\leq t_{i_{0}-1} } \left|\Phi_{\ep_{i_{0}-1}}f(y)-\Phi_{\ep_{i_{0}}}f(y) \right|\right.
	\\
	 &\left.\quad\quad\quad\quad\quad\quad\quad\quad=\mathop{\sup}\limits_{r<\ep_{i_{0}}< \ep_{i_{0}-1}\leq t_{i_{0}-1} } \left|\Phi_{\ep_{i_{0}-1}}f(y)-\Phi_{\ep_{i_{0}}}f(y) \right|\right\},	
\end{align*}
and
\begin{align*}
	B_{3}:&= \left\{y\in B:\mathop{\sup}\limits_{t_{i_{0}}\leq \ep_{i_{0}}< \ep_{i_{0}-1}\leq t_{i_{0}-1} } \left|\Phi_{\ep_{i_{0}-1}}f(y)-\Phi_{\ep_{i_{0}}}f(y) \right|\right.
	\\
	 &\left.\quad\quad\quad\quad\quad\quad\quad\quad=\mathop{\sup}\limits_{t_{i_{0}}\leq \ep_{i_{0}}\leq r< \ep_{i_{0}-1}\leq t_{i_{0}-1}} \left|\Phi_{\ep_{i_{0}-1}}f(y)-\Phi_{\ep_{i_{0}}}f(y) \right|\right\}.
\end{align*}
Then, we can write
\begin{align*}
F&=\sum\limits_{i=1}^{3}\int_{B_{i}}\left[  \mathcal{O}_{\{t_{i}\}_{i\in\mathbb{N}}}\left(\{\Phi_{t}\}_{t>0}, f \right)(x)-\mathop{\einf}\limits_{y\in B} \mathcal{O}_{\{t_{i}\}_{i\in\mathbb{N}}}\left(\{\Phi_{t}\}_{t>0},f \right)(y)\right] dx\\
&=:F_{1}+F_{2}+F_{3}.
\end{align*}

For $F_{1}$, we have,  for any $y\in B$,
$$\mathcal{O}_{\{t_{i}\}_{i\in\mathbb{N}}}\left(\{\Phi_{t}\}_{t>0},f \right)(y)\geq
\left ( \sum\limits_{i=1}^{i_{0}-2}\sup\limits_{t_{i+1}\leq \varepsilon_{i+1}<\varepsilon_{i}\leq t_{i}}\left| \Phi_{\ep_{i}}f(y)-\Phi_{\ep_{i+1}}f(y)\right|^{2}
\right )^{\frac{1}{2}},$$
and for any $x\in B_{1}$,
\begin{align*}
\mathcal{O}_{\{t_{i}\}_{i\in\mathbb{N}}}\left(\{\Phi_{t}\}_{t>0},f \right)(x)&\leq\left ( \sum\limits_{i=1}^{i_{0}-2}\sup\limits_{t_{i+1}\leq \varepsilon_{i+1}<\varepsilon_{i}\leq t_{i}}\left| \Phi_{\ep_{i}}f(x)-\Phi_{\ep_{i+1}}f(x)\right|^{2}\right )^{\frac{1}{2}}\\
&\quad+\left ( \sum\limits_{i=i_{0}-1}^{\infty}\sup\limits_{\substack{t_{i+1}\leq \varepsilon_{i+1}<\varepsilon_{i}\leq t_{i}\\ \ep_{i}\leq  r}}\left| \Phi_{\ep_{i}}f(x)-\Phi_{\ep_{i+1}}f(x)\right|^{2}\right )^{\frac{1}{2}}.
\end {align*}
Then, we get
\begin{align*}
F_{1}&\leq\int_{B_{1}}\left ( \sum\limits_{i=i_{0}-1}^{\infty}\sup\limits_{\substack{t_{i+1}\leq \varepsilon_{i+1}<\varepsilon_{i}\leq t_{i}\\ \ep_{i}\leq  r}}\left| \Phi_{\ep_{i}}f(x)-\Phi_{\ep_{i+1}}f(x)\right|^{2}\right )^{\frac{1}{2}}dx\\
&\quad +\left|B_{1} \right|\Bigg\{\mathop{\esup}\limits_{x,y\in B}\sum\limits_{i=1}^{i_{0}-2}\sup\limits_{t_{i+1}\leq \varepsilon_{i+1}<\varepsilon_{i}\leq t_{i} }\left| \left[ \Phi_{\ep_{i}}f(x)-\Phi_{\ep_{i+1}}f(x)\right]\right. \\
&\left.\quad\quad\quad\quad\quad-\left[  \Phi_{\ep_{i}}f(y)-\Phi_{\ep_{i+1}}f(y)\right]\right| ^{2}
\Bigg\}^{\frac1{2}}.
\end {align*}

Similarly, for $F_{2}$, we have,  for any $y\in B$,
$$\mathcal{O}_{\{t_{i}\}_{i\in\mathbb{N}}}\left(\{\Phi_{t}\}_{t>0},f \right)(y)\geq
\left ( \sum\limits_{i=1}^{i_{0}-1}\sup\limits_{t_{i+1}\leq \varepsilon_{i+1}<\varepsilon_{i}\leq t_{i}}\left| \Phi_{\ep_{i}}f(y)-\Phi_{\ep_{i+1}}f(y)\right|^{2}
\right )^{\frac{1}{2}}.$$
And for any $x\in B_{2}$,
\begin{align*}
	 \mathcal{O}_{\{t_{i}\}_{i\in\mathbb{N}}}\left(\{\Phi_{t}\}_{t>0},f \right)(x)
	&\leq\left ( \sum\limits_{i=1}^{i_{0}-1}\sup\limits_{\substack{t_{i+1}\leq \varepsilon_{i+1}<\varepsilon_{i}\leq t_{i}\\ \ep_{i+1}> r}}\left| \Phi_{\ep_{i}}f(x)-\Phi_{\ep_{i+1}}f(x)\right|^{2}\right )^{\frac{1}{2}}\\
	&\quad+\left ( \sum\limits_{i=i_{0}}^{\infty}\sup\limits_{t_{i+1}\leq \varepsilon_{i+1}<\varepsilon_{i}\leq t_{i}}\left| \Phi_{\ep_{i}}f(x)-\Phi_{\ep_{i+1}}f(x)\right|^{2}\right )^{\frac{1}{2}}.
\end {align*}
Then we obtain
\begin{align*}
	F_{2}&\leq\int_{B_{2}}\left ( \sum\limits_{i=i_{0}}^{\infty}\sup\limits_{t_{i+1}\leq \varepsilon_{i+1}<\varepsilon_{i}\leq t_{i}}\left| \Phi_{\ep_{i}}f(x)-\Phi_{\ep_{i+1}}f(x)\right|^{2}\right )^{\frac{1}{2}}dx\\
	&\quad +\left|B_{2} \right|\Bigg\{\mathop{\esup}\limits_{x,y\in B}\sum\limits_{i=1}^{i_{0}-1}\sup\limits_{\substack{t_{i+1}\leq \varepsilon_{i+1}<\varepsilon_{i}\leq t_{i}\\ \ep_{i+1}\geq r}}\left| \left[ \Phi_{\ep_{i}}f(x)-\Phi_{\ep_{i+1}}f(x)\right]\right. \\
	&\left.\quad\quad\quad\quad\quad-\left[  \Phi_{\ep_{i}}f(y)-\Phi_{\ep_{i+1}}f(y)\right]\right| ^{2}
	\Bigg\}^{\frac1{2}}.
	\end {align*}
	
For $F_{3}$, we have,  for any $y\in B$,
\begin{align*}
\mathcal{O}_{\{t_{i}\}_{i\in\mathbb{N}}}\left(\{\Phi_{t}\}_{t>0},f \right)(y)&\geq
\Bigg( \sum\limits_{i=1}^{i_{0}-2}\sup\limits_{t_{i+1}\leq \varepsilon_{i+1}<\varepsilon_{i}\leq t_{i}}\left| \Phi_{\ep_{i}}f(y)-\Phi_{\ep_{i+1}}f(y)\right|^{2}\\
& \quad\quad + \mathop{\sup}\limits_{t_{i_{0}}\leq \ep_{i_{0}}\leq r< \ep_{i_{0}-1}\leq t_{i_{0}-1} } \left|\Phi_{\ep_{i_{0}-1}}f(y)-\Phi_{\ep_{i_{0}}}f(y) \right|^{2}\Bigg)^{\frac{1}{2}}\\
&\geq
\Bigg( \sum\limits_{i=1}^{i_{0}-2}\sup\limits_{t_{i+1}\leq \varepsilon_{i+1}<\varepsilon_{i}\leq t_{i}}\left| \Phi_{\ep_{i}}f(y)-\Phi_{\ep_{i+1}}f(y)\right|^{2}\\
& \quad\quad + \mathop{\sup}\limits_{r< \ep_{i_{0}-1}\leq t_{i_{0}-1} } \left|\Phi_{\ep_{i_{0}-1}}f(y)-\Phi_{r}f(y) \right|^{2}\Bigg)^{\frac{1}{2}}.
\end {align*}
And for any $x\in B_{3}$,
\begin{align*}
&\mathcal{O}_{\{t_{i}\}_{i\in\mathbb{N}}}\left(\{\Phi_{t}\}_{t>0},f \right)(x)\\
&\quad \leq\Bigg( \sum\limits_{i=i_{0}}^{\infty}\sup\limits_{t_{i+1}\leq \varepsilon_{i+1}<\varepsilon_{i}\leq t_{i} }\left| \Phi_{\ep_{i}}f(x)-\Phi_{\ep_{i+1}}f(x)\right|^{2}\\
&\quad\quad\quad + \mathop{\sup}\limits_{t_{i_{0}}\leq \ep_{i_{0}}\leq r } \left|\Phi_{\ep_{i_{0}}}f(x)-\Phi_{r}f(x) \right|^{2}\Bigg)^{\frac{1}{2}}\\
&\quad\quad\quad  +
\Bigg( \sum\limits_{i=1}^{i_{0}-2}\sup\limits_{t_{i+1}\leq \varepsilon_{i+1}<\varepsilon_{i}\leq t_{i}}\left| \Phi_{\ep_{i}}f(x)-\Phi_{\ep_{i+1}}f(x)\right|^{2}\\
& \quad\quad\quad\quad\quad + \mathop{\sup}\limits_{r< \ep_{i_{0}-1}\leq t_{i_{0}-1} } \left|\Phi_{\ep_{i_{0}-1}}f(x)-\Phi_{r}f(x) \right|^{2}\Bigg)^{\frac{1}{2}}\\
&\quad \leq\Bigg( \sum\limits_{i=i_{0}-1}^{\infty}\sup\limits_{\substack{t_{i+1}\leq \varepsilon_{i+1}<\varepsilon_{i}\leq t_{i}\\ \ep_{i}\leq r}}\left| \Phi_{\ep_{i}}f(x)-\Phi_{\ep_{i+1}}f(x)\right|^{2}\Bigg)^{\frac{1}{2}}\\
&\quad\quad\quad + \Bigg(\mathop{\sup}\limits_{r<\ep_{i_{0}-1}\leq t_{i_{0}-1} } \left|\Phi_{\ep_{i_{0}-1}}f(x)-\Phi_{r}f(x) \right|^{2}\\
&\quad\quad\quad\quad\quad  + \sum\limits_{i=1}^{i_{0}-2}\sup\limits_{t_{i+1}\leq \varepsilon_{i+1}<\varepsilon_{i}\leq t_{i}}\left| \Phi_{\ep_{i}}f(x)-\Phi_{\ep_{i+1}}f(x)\right|^{2}\Bigg)^{\frac{1}{2}}.
\end {align*}
Furthermore, we obtain
\begin{align*}
F_{3}&\leq\int_{B_{3}}\left(     \sum\limits_{i=i_{0}-1}^{\infty}\sup\limits_{\substack{t_{i+1}\leq \varepsilon_{i+1}<\varepsilon_{i}\leq t_{i}\\ \ep_{i}\leq r}}\left| \Phi_{\ep_{i}}f(x)-\Phi_{\ep_{i+1}}f(x)\right|^{2}        \right)^{\frac{1}{2}}dx\\
&\quad +\left| B_{3}\right| \mathop{\esup}\limits_{x,y\in B}\Bigg\{ \sum\limits_{i=1}^{i_{0}-2}\sup\limits_{t_{i+1}\leq \varepsilon_{i+1}<\varepsilon_{i}\leq t_{i}}\Big|\left[ \Phi_{\ep_{i}}f(x)-\Phi_{\ep_{i+1}}f(x)\right]\\
&\quad\quad\quad\quad\quad\quad\quad\quad-\left[  \Phi_{\ep_{i}}f(y)-\Phi_{\ep_{i+1}}f(y)\right]\Big| ^{2}
\\
&\quad\quad\quad\quad\quad\quad\quad\quad\quad+\mathop{\sup}\limits_{r\leq \ep_{i_{0}}< \ep_{i_{0}-1}\leq t_{i_{0}-1} } \left|\left[ \Phi_{\ep_{i_{0}-1}}f(x)-\Phi_{\ep_{i_{0}}}f(x)\right] \right.\\
&\quad\quad\quad\quad\quad\quad\quad\quad\quad\quad\left.-\left[ \Phi_{\ep_{i_{0}-1}}f(y)-\Phi_{\ep_{i_{0}}}f(y)\right]  \right|^{2}\Bigg\}^{\frac1{2}}\\
&\leq\int_{B_{3}}\bigg(     \sum\limits_{i=i_{0}-1}^{\infty}\sup\limits_{\substack{t_{i+1}\leq \varepsilon_{i+1}<\varepsilon_{i}\leq t_{i}\\ \ep_{i}\leq r}}\left| \Phi_{\ep_{i}}f(x)-\Phi_{\ep_{i+1}}f(x)\right|^{2}        \bigg)^{\frac{1}{2}}dx\\
&\quad +\left| B_{3}\right| \mathop{\esup}\limits_{x,y\in B}\Bigg\{ \sum\limits_{i=1}^{i_{0}-1}\sup\limits_{\substack{t_{i+1}\leq \varepsilon_{i+1}<\varepsilon_{i}\leq t_{i}\\ \ep_{i+1}\geq r}}\Big|\left[ \Phi_{\ep_{i}}f(x)-\Phi_{\ep_{i+1}}f(x)\right]\\
&\quad\quad\quad\quad\quad\quad\quad\quad-\left[ \Phi_{\ep_{i}}f(y)-\Phi_{\ep_{i+1}}f(y)\right]\Big|^{2}\Bigg\}^{\frac1{2}}.
\end {align*}
\par Based on the above arguments, we obtain
\begin{align*}
&\left\| \mathcal{O}_{\{t_{i}\}_{i\in\mathbb{N}}}\left(\{\Phi_{t}\}_{t>0}, f \right)\right\| _{\bloz} \leq \frac{F_{1}+F_{2}+F_{3}}{\left|B\right|}\\
&\quad \leq \frac1{\left|B\right|}\int_{B}\left( \sum\limits_{i=i_{0}-1}^{\infty}\sup\limits_{\substack{t_{i+1}\leq \varepsilon_{i+1}<\varepsilon_{i}\leq t_{i}\\ \ep_{i}\leq r}}\left| \Phi_{\ep_{i}}f(x)-\Phi_{\ep_{i+1}}f(x)\right|^{2}\right)^{\frac{1}{2}}dx\\
&\quad\quad +\mathop{\esup}\limits_{x,y\in B}\Bigg\{ \sum\limits_{i=1}^{i_{0}-1}\sup\limits_{\substack{t_{i+1}\leq \varepsilon_{i+1}<\varepsilon_{i}\leq t_{i}\\ \ep_{i+1}\geq r}}\Big|\left[ \Phi_{\ep_{i}}f(x)-\Phi_{\ep_{i+1}}f(x)\right]\\
&\quad\quad\quad\quad\quad\quad\quad-\left[ \Phi_{\ep_{i}}f(y)-\Phi_{\ep_{i+1}}f(y)\right]\Big|^{2}\Bigg\}^{\frac1{2}}\\
&\quad=:E_{1}+E_{2}.
\end {align*}

Similar to the estimates of $G_{1}$ and $G_{2}$ in Case 1, it follows that $$E_{1}\leq C\|f\|_{\bmoz}.$$ Then we only need to estimate $E_{2}$. Since $\int_{\rn}\phi_{t}(x)dx=1$, then by Newton-Leibniz formula,
\begin{align*}
& \Big|  \left[ \Phi_{\ep_{i}}f(x)-\Phi_{\ep_{i+1}}f(x)\right]-\left[ \Phi_{\ep_{i}}f(y)-\Phi_{\ep_{i+1}}f(y)\right]  \Big| \\
&\quad = \left|\int_{\ep_{i+1}}^{\ep_{i}} \frac{\partial}{\partial t} \left[ \Phi_{t}(f-f_{B})(x)-\Phi_{t}(f-f_{B})(y)\right]dt\right| .
\end {align*}
Futhermore, we obtain
\begin{align*}
E_{2}&\leq \mathop{\esup}\limits_{x,y\in B}\Bigg\{ \sum\limits_{i=1}^{i_{0}-1}\sup\limits_{\substack{t_{i+1}\leq \varepsilon_{i+1}<\varepsilon_{i}\leq t_{i}\\ \ep_{i+1}\geq r} }\Big|\left[ \Phi_{\ep_{i}}f(x)-\Phi_{\ep_{i+1}}f(x)\right]-\left[ \Phi_{\ep_{i}}f(y)-\Phi_{\ep_{i+1}}f(y)\right]\Big|\Bigg\}\\
&\leq \mathop{\esup}\limits_{x,y\in B} \int_{r}^{\infty} \left| \frac{\partial}{\partial t} \left[ \Phi_{t}(f-f_{B})(x)-\Phi_{t}(f-f_{B})(y)\right]  \right| dt\\
& \leq \mathop{\esup}\limits_{x,y\in B}\int_{r}^{\infty}\left[ \int_{2B}+\int_{(2B)^c}\right] \left|f(z)-f_{B} \right|\left|\frac{\partial}{\partial t}\left( \phi_{t}(x-z)-\phi_{t}(y-z)\right) \right|dzdt\\
&=: E_{21}+E_{22}.
\end {align*}

From Lemma \ref{BMO pro} (i) and \eqref{three inner partial},  we deduce
\begin{align*}
E_{21}&\leq C \mathop{\esup}\limits_{x,y\in B}\int_{r}^{\infty}\int_{2B}\left|f(z)-f_{B} \right|\frac{\left|x-y \right| }{t^{n+2}}dzdt\\
&\leq \frac{C}{r^{n}}\int_{2B}\left|f(z)-f_{B} \right|dz\leq C\|f\| _{\bmoz}.
\end {align*}

 From Lemma \ref{BMO pro} (iii) and \eqref{three outer partial}, we can obtain
\begin{align*}
E_{22}&\leq C\mathop{\esup}\limits_{x,y\in B}\int_{r}^{\infty}\int_{(2B)^c}\left|f(z)-f_{B}\right|\frac{\left|x-y \right| }{t^{\frac{3}{2}}\left|z-x_{0} \right|^{n+\frac1{2}}}dzdt\\
&\leq C \int_{r}^{\infty}\frac{r}{t^{\frac{3}{2}}}dt\int_{(2B)^c}\left|f(z)-f_{B}\right|\frac{1}{\left|z-x_{0} \right|^{n+\frac1{2}}}dz\\
&\leq C \sqrt{r}\sum\limits_{k=1}^{\infty}\int_{2^{k+1}B\setminus 2^{k}B }\left|f(z)-f_{B} \right|\frac{1}{\left(2^{k}r\right) ^{n+\frac1{2}}}dz\\
&\leq C \sum\limits_{k=1}^{\infty}\int_{2^{k+1}B}\left|f(z)-f_{B} \right|\frac1{\left(2^{k+1} r\right)^{n} \left(\sqrt{2} \right)^{k} }dz\\
&\leq C \sum\limits_{k=1}^{\infty}\frac1{\left(\sqrt{2} \right)^{k}}\frac1{\left|2^{k+1}B \right|  }\int_{2^{k+1}B}\left|f(z)-f_{B} \right|dz\\
&\leq C \sum\limits_{k=1}^{\infty}\frac{k+1}{\left(\sqrt{2} \right)^{k}}\|f\|_{\bmoz}\leq C\|f\|_{\bmoz}.
\end {align*}

Thus, we have $E_{2}\leq C\|f\|_{\bmoz}.$ This along with the estimate of $E_{1}$ yields \eqref{main result O}. Accordingly, we finish the proof of Theorem \ref{main O}.
\end{proof}

\begin{proof}[\bf Proof of Theorem \ref{main V}]
We first claim that if there exists a point $x_{0}\in\rn$ such that
\begin{equation}\label{Vae}
	\mathcal{V}_{\rho}(\{\Phi_{t}\}_{t>0},f)(x_0)<\infty,
\end{equation}
then $\mathcal{V}_{\rho}(\{\Phi_{t}\}_{t>0},f)$ is finite almost everywhere. As a matter of fact, in this case, we will show there  exists $C>0$  depending on $n$ and $\phi$ such that for any ball $B:=B(x_0, r)$ with $r>0$,
	\begin{equation}\label{V a.e}
		E:=\frac1{|B|}\int_{B}\lf [ \mathcal{V}_{\rho}(\{\Phi_{t}\}_{t>0},f)(x)-\mathop{\inf}\limits_{y\in B} \mathcal{V}_{\rho}(\{\Phi_{t}\}_{t>0},f)(y)\r ]dx\leq C \|f\|_{\bmoz}.
	\end{equation}
	
For $r>0$, we set
\begin{align*}
	&M_{r}:=\{\{\ep_{i}\}_{i=1}^{\infty}: \{\ep_{i}\}_{i=1}^{\infty} \ \text{is a sequence of positive number decreasing to zero with} \ \ep_{1}\leq r\},\\
	&N_{r}:=\{\{\eta_{j}\}_{j=1}^{k}: \{\eta_{j}\}_{j=1}^{k} \ \text{is a finite decreasing sequence of with the last item} \ \eta_{k}\geq r\}.
\end{align*}
And let
\begin{align*}
&B_{1}:=\left\lbrace x\in B: \mathcal{V}_{\rho}(\{\Phi_{t}\}_{t>0},f)(x)= \sup\limits_{\{\ep_{i}\}_{i=1}^{\infty}\in M_{r}}\lf ( \sum\limits_{i=1}^{\infty}|{\Phi_{\ep_{i}}f(x)-\Phi_{\ep_{i+1}}f(x)}|^{\rho}\r)^{\frac{1}{\rho}}\right\rbrace,\\
&B_{2}:=B\setminus B_1.
\end{align*}
 Then we derive
 \begin{align*}
 	E&\leq \frac1{|B|}\int_{B_{1}}\sup\limits_{\{\ep_{i}\}_{i=1}^{\infty}\in M_{r}}\lf ( \sum\limits_{i=1}^{\infty}|{\Phi_{\ep_{i}}f(x)-\Phi_{\ep_{i+1}}f(x)}|^{\rho}\r)^{\frac{1}{\rho}}dx\\
 	&\quad +\frac1{|B|}\int_{B_{2}}\lf [ \mathcal{V}_{\rho}(\{\Phi_{t}\}_{t>0},f)(x)-\mathop{\inf}\limits_{y\in B} \mathcal{V}_{\rho}(\{\Phi_{t}\}_{t>0},f)(y)\r ]dx\\
 	&=: E_{1}+E_{2}.
 \end{align*}

Similar to the estimation of $G$ in Theorem \ref{main O}, we obtain
$$E_{1}\leq C \|f\|_{\bmoz}.$$

For $E_{2}$, notice that for any $x\in B_{2}$,
\begin{align*}
	\mathcal{V}_{\rho}(\{\Phi_{t}\}_{t>0},f)(x)\leq& \sup\limits_{\{\ep_{i}\}_{i=1}^{\infty}\in M_{r}}\lf ( \sum\limits_{i=1}^{\infty}|{\Phi_{\ep_{i}}f(x)-\Phi_{\ep_{i+1}}f(x)}|^{\rho}\r)^{\frac{1}{\rho}}\\
	&\quad+\sup\limits_{\{\eta_{j}\}_{j=1}^{k}\in N_{r}}\lf ( \sum\limits_{j=1}^{k-1}|{\Phi_{\eta_{j}}f(x)-\Phi_{\eta_{j+1}}f(x)}|^{\rho}\r)^{\frac{1}{\rho}}.
\end{align*}
And for any $y\in B$,
$$\mathcal{V}_{\rho}(\{\Phi_{t}\}_{t>0},f)(y)\geq \sup\limits_{\{\eta_{j}\}_{j=1}^{k}\in N_{r}}\lf ( \sum\limits_{j=1}^{k-1}|{\Phi_{\eta_{j}}f(y)-\Phi_{\eta_{j+1}}f(y)}|^{\rho}\r)^{\frac{1}{\rho}} .$$
Then we obtain
\begin{align*}
	E_{2}&\leq \frac1{|B|}\int_{B_{2}}\sup\limits_{\{\ep_{i}\}_{i=1}^{\infty}\in M_{r}}\lf ( \sum\limits_{i=1}^{\infty}|{\Phi_{\ep_{i}}f(x)-\Phi_{\ep_{i+1}}f(x)}|^{\rho}\r)^{\frac{1}{\rho}}dx\\
	&\quad +\sup\limits_{x, y\in B}\sup \limits_{\{\ep_{j}\}_{j=1}^{k}}\left[ \lf ( \sum\limits_{j=1}^{k-1}|{\Phi_{\ep_{j}}f(x)-\Phi_{\ep_{j+1}}f(x)}|^{\rho}\r)^{\frac{1}{\rho}}\right.\\
	&\left. \quad\quad \quad\quad \quad\quad \quad\quad -\lf ( \sum\limits_{j=1}^{k-1}|{\Phi_{\ep_{j}}f(y)-\Phi_{\ep_{j+1}}f(y)}|^{\rho}\r)^{\frac{1}{\rho}}\right]\\
	&\leq  \frac1{|B|}\int_{B}\sup\limits_{\{\ep_{i}\}_{i=1}^{\infty}\in M_{r}} \lf ( \sum\limits_{i=1}^{\infty}|{\Phi_{\ep_{i}}f(x)-\Phi_{\ep_{i+1}}f(x)}|^{\rho}\r)^{\frac{1}{\rho}}dx\\
	&\quad +\sup\limits_{x,y\in B}\sup \limits_{\{\ep_{j}\}_{j=1}^{k}} \Bigg\{ \sum\limits_{j=1}^{k-1}\Big|\left[ \Phi_{\ep_{j}}f(x)-\Phi_{\ep_{j+1}}f(x)\right]\\
	&\quad\quad\quad\quad\quad\quad\quad\quad\ -\left[ \Phi_{\ep_{j}}f(y)-\Phi_{\ep_{j+1}}f(y)\right]\Big|\Bigg\}\\
	&=:E_{21}+E_{22}.
\end{align*}

Similar to the estimation of $E_1$ above, we obtain
$$E_{21}\leq C \|f\|_{\bmoz}.$$

By proceeding as in the  estimations of $E_{2}$ in the proof of Theorem \ref{main O}, we can see that
$$E_{22}\leq C \|f\|_{\bmoz}.$$

 Thus, we have $E_{2}\leq C \|f\|_{\bmoz}.$ This along with the estimate of $E_{1}$ yields \eqref{V a.e}. Furthermore, by \eqref{Vae}, we can see that $ \mathop{\inf}_{y\in B} \mathcal{V}_{\rho}(\{\Phi_{t}\}_{t>0},f)(y) <\infty.$ Then \eqref{V a.e} implies  $$\mathcal{V}_{\rho}(\{\Phi_{t}\}_{t>0},f)(x)<\infty,a.e.\ x\in B(x_0, r).$$
 By the arbitrariness of $r$, we get
 $$\mathcal{V}_{\rho}(\{\Phi_{t}\}_{t>0},f)(x)<\infty,a.e.\ x\in \rn.$$

 Now we only need to show that,  in the case that $\mathcal{V}_{\rho}(\{\Phi_{t}\}_{t>0},f)$ is finite almost everywhere, there exists $C>0$ depending on $n$ and $\phi$ such that for any ball $B:=B(x_{0},r)\subset \rn$,
 \begin{equation}\label{main result V}
 	\frac1{|B|}\int_{B}\lf [ \mathcal{V}_{\rho}(\{\Phi_{t}\}_{t>0},f)(x)-\mathop{\einf}\limits_{y\in B} \mathcal{V}_{\rho}(\{\Phi_{t}\}_{t>0},f)(y)\r ]dx\leq C \|f\|_{\bmoz}.
 \end{equation}
 The proof of \eqref{main result V} is similar to \eqref{V a.e}, hence we omit the details, which completes the proof of Theorem \ref{main V}.

\end{proof}
\section{Boundedness of the maximal differential transform}\label{sec4}
In this section, we first prove the $L^{p}$ boundedness of the maximal differential transform $\mathcal{S}_{\boldsymbol{a,v};*}\left(\{\Phi_{t}\}_{t>0}, \cdot\right)$. We obtain this result by following some idea in \cite{CT}. Based on this result, we further establish the boundedness from $\bmoz$ into $\bloz$ of $\mathcal{S}_{\boldsymbol{a,v};*}\left(\{\Phi_{t}\}_{t>0}, \cdot\right)$.
\subsection{$L^{p}$ boundedness of $\mathcal{S}_{\boldsymbol{a,v};*}\left(\{\Phi_{t}\}_{t>0}, \cdot \right)$ }\label{subsec4.1}
Let $\mathcal{S}_{\boldsymbol{a,v};N}\left(\{\Phi_{t}\}_{t>0}, f \right)(x)$ be as in \eqref{partial sum}.
 We first establish the $L^2$ boundedness of it.
\begin{prop}\label{S L2}
	Let $\boldsymbol{a}$ be a $\delta$-lacunary sequence of positive numbers, $\boldsymbol{v}$ be a bounded sequence of complex numbers and $\phi\in \mathcal{S}(\rn)$ with $\int_{\rn}\phi(x)dx=1$. Then there is a constant $C>0$ depending on $n$, $\phi$ and $\left\| \boldsymbol{v}\right\|_{l^{\infty}(\mathbb{Z})} $ (not on $N$) such that
	$$\|\mathcal{S}_{\boldsymbol{a,v};N}\left(\{\Phi_{t}\}_{t>0}, f\right)\|_{L^2(\rn)}\leq C\|f\|_{L^2(\rn)},$$
	for all $f\in L^2(\rn).$
\end{prop}
\begin{proof}
	For $f\in L^{2}(\rn)$, we let $\mathcal{F}(f)$ denote the Fourier transform of $f$. By the Plancherel theorem, we have
\begin{align*} 	
	\|\mathcal{S}_{\boldsymbol{a,v};N}\left(\{\Phi_{t}\}_{t>0}, f\right)\|_{L^2(\rn)}^{2}
	 &=\left\| \sum\limits_{i=N_{1}}^{N_{2}}v_{i}\left(\Phi_{a_{i+1}}f(x)-\Phi_{a_{i}}f(x) \right)\right\| _{L^2(\rn)}^{2}\\
	&=\int_{\rn} \left| \sum\limits_{i=N_{1}}^{N_{2}} v_{i} \left[ \mathcal{F}\left( \Phi_{a_{i+1}}f\right) (\xi)-\mathcal{F}\left( \Phi_{a_{i}}f\right) (\xi) \right]  \right| ^{2}d\xi\\
	&\leq C\int_{\rn}\left\lbrace \sum\limits_{i=N_{1}}^{N_{2}}\left| \int _{a_{i}}^{a_{i+1}}\frac{\partial}{\partial t}\left[ \mathcal{F}\left( \Phi_{t}f\right)
	(\xi)\right] dt\right| \right\rbrace ^2 d\xi\\
	&\leq C \int_{\rn}\left\lbrace \int_{0}^{\infty}\left|\frac{\partial}{\partial t}\left[ \mathcal{F}\left( \phi_{t}\right)(\xi)  \mathcal{F}\left( f\right)
	(\xi)\right] \right| dt\right\rbrace ^2 d\xi\\
	&= C \int_{\rn}\left\lbrace \int_{0}^{\infty}\left|\frac{\partial}{\partial t}\left[ \mathcal{F}\left( \phi\right)(t\xi)  \right] \right| dt\right\rbrace ^2 \left[ \mathcal{F}\left( f\right)(\xi)\right] ^{2}d\xi.
\end{align*}
Since $g:=\mathcal{F}\left( \phi\right)\in \mathcal{S}(\rn)$, then we have
\begin{align*}
\left\lbrace \int_{0}^{\infty}\left|\frac{\partial}{\partial t}\left[ \mathcal{F}\left( \phi\right)(t\xi)  \right] \right| dt\right\rbrace ^2&=\left[  \int_{0}^{\infty}\left| \nabla g (t\xi) \cdot \xi  \right| dt\right] ^2\\
&\leq C\left| \xi\right| ^2\left[  \int_{0}^{\infty}\frac{1}{\left( 1+t\left| \xi\right|\right) ^2}dt\right] ^2\leq C.
\end{align*}
Hence, combining with previous estimates, we can see
$$ \|\mathcal{S}_{\boldsymbol{a,v};N}\left(\{\Phi_{t}\}_{t>0}, f\right)\|_{L^2(\rn)}^{2}\leq C\int_{\rn}\left[ \mathcal{F}\left( f\right)(\xi)\right] ^{2}d\xi= C\|f\|_{L^2(\rn)}^2.$$
\end{proof}
Next, we will prove the uniform boundedness of the operators $\mathcal{S}_{\boldsymbol{a,v};N}\left(\{\Phi_{t}\}_{t>0}, \cdot\right)$. The standard Calder\'{o}n-Zygmund theory will be a fundamental tool, for which the reader can see \cite[Theorem 5.10]{JWE}. By this theory, Proposition \ref{S L2}  and Lemma \ref{ S kernel es }, we obtain the uniform boundedness of $\mathcal{S}_{\boldsymbol{a,v};N}\left(\{\Phi_{t}\}_{t>0}, \cdot\right)$ as follows.
 \begin{prop}\label{S partial Lp}
 	Let $\boldsymbol{a}$ be a $\delta$-lacunary sequence of positive numbers, $\boldsymbol{v}$ be a bounded sequence of complex numbers and $\phi\in \mathcal{S}(\rn)$ with $\int_{\rn}\phi(x)dx=1$. For any $1<p<\infty$, there exists a constant $C$ depending on $n, p, \phi$ and $\left\| \boldsymbol{v}\right\|_{l^{\infty}(\mathbb{Z})} $ (not on $N$) such that
 	$$\|\mathcal{S}_{\boldsymbol{a,v};N}\left(\{\Phi_{t}\}_{t>0}, f\right)\|_{L^p(\rn)}\leq C\|f\|_{L^p(\rn)}.$$
 \end{prop}

The following lemma, parallel to \cite[Proposition 3.1]{CMT}, shows that, without loss of generality, we may assume the $\delta$-lacunary sequence $\boldsymbol{a}:=\{a_{i}\}_{i\in \mathbb{Z}}$ satisfies \eqref{add}.
\begin{lem}\label{addtion}
Let $\phi\in \mathcal{S}(\rn)$ with $\int_{\rn}\phi(x)dx=1$. Given  a $\delta$-lacunary sequence of positive numbers $\boldsymbol{a}:=\{a_{i}\}_{i\in\mathbb{Z}}$ and a bounded  sequence of complex numbers $\boldsymbol{v}:=\{v_{i}\}_{i\in\mathbb{Z}}$, we can find a $\delta$-lacunary sequence of positive numbers $\boldsymbol{b}:=\{b_{i}\}_{i\in\mathbb{Z}}$ and a bounded sequence of complex numbers $\boldsymbol{w}:=\{w_{i}\}_{i\in\mathbb{Z}}$ verifying the following properties:
\begin{itemize}
	\item [{\rm(i)}] $1<\delta\leq \frac{b_{i+1}}{b_{i}}\leq \delta^{2}, i\in\mathbb{Z}, \left\| \boldsymbol{w}\right\|_{l^{\infty}(\mathbb{Z})}=\left\| \boldsymbol{v}\right\|_{l^{\infty}(\mathbb{Z})}$,
	\item [{\rm(ii)}] For any $N=(N_{1},N_{2})\in \mathbb{Z}_{<}^{2},$ there exists $N'=({N'_{1}},{N'_{2}})\in \mathbb{Z}_{<}^{2}$ such that\ $$\mathcal{S}_{\boldsymbol{a,v};N}\left(\{\Phi_{t}\}_{t>0}, \cdot\right)=\mathcal{S}_{\boldsymbol{b,w};N'}\left(\{\Phi_{t}\}_{t>0}, \cdot\right).$$
\end{itemize}
\end{lem}

Next,  we will establish the following Cotlar's type inequality to control $\mathcal{S}_{\boldsymbol{a,v};M}^{\ast}\left(\{\Phi_{t}\}_{t>0}, \cdot\right),$ where $M>2$ is a positive integer and
$$\mathcal{S}_{\boldsymbol{a,v};M}^{\ast}\left(\{\Phi_{t}\}_{t>0}, f\right)(x):=\sup\limits_{-M+1<N_{1}<N_{2}<M-1}\left|\mathcal{S}_{\boldsymbol{a,v};(N_{1},N_{2})}\left(\{\Phi_{t}\}_{t>0}, f \right)(x) \right| .$$
\begin{thm}\label{C ineq}
	Let $\boldsymbol{a}$ be a $\delta$-lacunary sequence of positive numbers satisfying $\eqref{add}$, $\boldsymbol{v}$ be a bounded sequence of complex numbers and $\phi\in \mathcal{S}(\rn)$ with $\int_{\rn}\phi(x)dx=1$. For each $q\in(1,\infty)$, there exists a constant $C$ depending on $n,\phi, \delta, q$ and $\left\| \boldsymbol{v}\right\|_{l^{\infty}(\mathbb{Z})} $ such that, for every
	$x\in\rn$ and every positive integer $ M>2$,
	 $$\mathcal{S}_{\boldsymbol{a,v};M}^{\ast}\left(\{\Phi_{t}\}_{t>0}, f\right)(x)\leq C\left\lbrace\mathcal{M}\left( \mathcal{S}_{\boldsymbol{a,v};(-M,M)}\left(\{\Phi_{t}\}_{t>0}, f\right)\right)(x)+\mathcal{M}_{q}f(x) \right\rbrace, $$
where $\mathcal{M}$ is the Hardy-Littlewood maximal operator and
$$\mathcal{M}_{q}f(x):=\sup\limits_{r>0}\left(\frac1{\left| B(x,r)\right|}\int_{B(x,r)}\left|f(y) \right|^q dy  \right)^{\frac1{q}},\quad 1<q<\infty.$$
\end{thm}
For the proof of this theorem we shall need the folllowing lemma:
\begin{lem}\label{for C ineq}
 Let $\phi\in \mathcal{S}(\rn)$ with $\int_{\rn}\phi(x)dx=1$, $\boldsymbol{a}:=\{a_{i}\}_{i\in \mathbb{Z}}$ be a $\delta$-lacunary sequence satisfying $\eqref{add}$ and $\boldsymbol{v}:=\{v_{i}\}_{i\in \mathbb{Z}}$ be a bounded sequence of complex numbers. For
 every positive integer $ M>2$, and integer $m$ satisfying $\left| m\right|< M$, there exists a positive $C$ depending on $n,\phi, \delta$ and $\left\| \boldsymbol{v}\right\|_{l^{\infty}(\mathbb{Z})} $ such that
 \begin{itemize}
 	\item [{\rm(i)}] for any $x, y\in \rn$,
 	\begin{equation*}
 		\left|\sum\limits_{i=m}^{M} v_{i}\left[ \phi_{a_{i+1}}(x-y)-\phi_{a_{i}}(x-y)\right] \right| \leq \frac{C}{a_{m}^n};
 	\end{equation*}
 	\item [{\rm(ii)}] if $k\geq m$ and $z,y \in \rn$ with $\left| z-y\right|\geq a_{k}$, then
 	\begin{equation*}
 	\left|\sum\limits_{i=-M}^{m-1} v_{i}\left[ \phi_{a_{i+1}}(z-y)-\phi_{a_{i}}(z-y)\right] \right| \leq C\frac{\delta^{m-k+1}}{a_{k}^n}	.
 	\end{equation*}
 \end{itemize}
\end{lem}
\begin{proof}
	For (i), by the mean value theorem, \eqref{add} and \eqref{one partial}, there exists $\xi_{i}\in(a_{i},a_{i+1})$ such that
	\begin{align*} 	
	&\left|\sum\limits_{i=m}^{M} v_{i}\left[ \phi_{a_{i+1}}(x-y)-\phi_{a_{i}}(x-y)\right] \right|\\
	&\quad \leq \left\| \boldsymbol{v}\right\|_{l^{\infty}(\mathbb{Z})}	 \sum\limits_{i=m}^{M}\left( a_{i+1}-a_{i}\right) \left| \frac{\partial}{\partial t} \left( \phi_{t}(x-y)\right)|_{t=\xi_{i}} \right| \\
	&\quad\leq C\sum\limits_{i=m}^{M}\left( a_{i+1}-a_{i}\right)\left| \frac{ 1}{(\xi_{i}+\left|x-y \right|)^{n+1}}\right|\\
	&\quad\leq C \sum\limits_{i=m}^{M}\left( \delta^{2}-1\right)\frac{a_{i}}{a_{i}^{n+1}}\leq C \sum\limits_{i=m}^{M}\frac{\delta^{2}-1}{a_{m}^{n}\delta^{\left(i-m \right)n }}\leq \frac{C}{a_{m}^n}.
	\end{align*}

	Now we shall prove (ii). By the mean value theorem and \eqref{add}, there exists $\xi_{i}\in(a_{i},a_{i+1})$ such that
	\begin{align*}
		&\left|\sum\limits_{i=-M}^{m-1} v_{i}\left[ \phi_{a_{i+1}}(z-y)-\phi_{a_{i}}(z-y)\right] \right| \\
		&\quad \leq \left\| \boldsymbol{v}\right\|_{l^{\infty}(\mathbb{Z})}	 \sum\limits_{i=-M}^{m-1}\left( a_{i+1}-a_{i}\right) \left| \frac{\partial}{\partial t} \left( \phi_{t}(z-y)\right)|_{t=\xi_{i}} \right| \\
		&\quad\leq C\sum\limits_{i=-M}^{m-1}\left( \delta^{2}-1\right)\left| \frac{a_{i}}{(\xi_{i}+\left|z-y \right|)^{n+1}}\right|\\
		&\quad\leq C\left( \sum\limits_{i=-M}^{m-1}\frac{a_{i}}{a_{k}}\right) \frac1{{a_{k}}^n}\leq C \left( \sum\limits_{i=-M}^{m-1}\frac{1}{\delta^{k-i}}\right)\frac1{{a_{k}}^n}\leq C\frac{\delta^{m-k+1}}{a_{k}^n},
	\end{align*}
where we have used that $k\geq m$ in the last inequality.
\end{proof}
Now, we are in a position to prove  Theorem \ref{C ineq}.
\begin{proof}[\bf Proof of Theorem \ref{C ineq}]
Notice that, for any $x_{0}\in \rn$ and $N:=(N_{1},N_{2})$ with  $-M+1<N_{1}<N_{2}<M-1$,
$$\mathcal{S}_{\boldsymbol{a,v};N}\left(\{\Phi_{t}\}_{t>0}, f\right)(x_{0})=\mathcal{S}_{\boldsymbol{a,v};(N_{1},M)}\left(\{\Phi_{t}\}_{t>0}, f\right)(x_{0})-\mathcal{S}_{\boldsymbol{a,v};(N_{2}+1,M)}\left(\{\Phi_{t}\}_{t>0}, f\right)(x_{0}).$$
Then it suffices to estimate $\left| \mathcal{S}_{\boldsymbol{a,v};(m,M)}\left(\{\Phi_{t}\}_{t>0}, f\right)(x_{0})\right| $ for $\left| m\right|< M$ with constants independent of $m$ and $M$. Denote $B_{k}:=B(x_{0},a_{k})$ for each $k\in \mathbb{N}$. We split $f$ as
$$f=f\chi_{B_{m}}+ f\chi_{\left( B_{m}\right) ^c}=:f_{1}+f_{2}.$$
Then we have
\begin{align*} \left| \mathcal{S}_{\boldsymbol{a,v};(m,M)}\left(\{\Phi_{t}\}_{t>0}, f\right)(x_{0})\right|&\leq \left| \mathcal{S}_{\boldsymbol{a,v};(m,M)}\left(\{\Phi_{t}\}_{t>0}, f_{1}\right)(x_{0})\right|+\left| \mathcal{S}_{\boldsymbol{a,v};(m,M)}\left(\{\Phi_{t}\}_{t>0}, f_{2}\right)(x_{0})\right|\\
	&=:A_{1}+A_{2}.
\end{align*}

For $A_{1}$, by Lemma \ref{for C ineq} (i) and H\"older's inequality,
\begin{align*}
A_{1} &=\left|\int_{\rn}\sum\limits_{i=m}^{M} v_{i}\left[ \phi_{a_{i+1}}(x-y)-\phi_{a_{i}} (x-y)\right]f_{1}(y)dy \right|\\
&\leq \frac{C}{a_{m}^n}\int_{\rn}\left| f_{1}(y)\right| dy\leq C\mathcal{M}(f)(x_{0})\leq C\mathcal{M}_{q}f(x_{0}).
\end{align*}

For $A_{2}$,
\begin{align*}
A_{2}&=\frac{C}{\left(a_{m-1}/2\right)^{n} }\int_{B(x_{0},a_{m-1}/2)}\left| \mathcal{S}_{\boldsymbol{a,v};(m,M)}\left(\{\Phi_{t}\}_{t>0}, f_{2}\right)(x_{0})\right|dz \\
&\leq \frac{C}{\left(a_{m-1}/2\right)^{n} }\int_{B(x_{0},a_{m-1}/2)}\left|\mathcal{S}_{\boldsymbol{a,v};(m,M)}\left(\{\Phi_{t}\}_{t>0}, f_{2}\right)(z)- \mathcal{S}_{\boldsymbol{a,v};(m,M)}\left(\{\Phi_{t}\}_{t>0}, f_{2}\right)(x_{0})\right|dz\\
&\quad \quad +\frac{C}{\left(a_{m-1}/2\right)^{n} }\int_{B(x_{0},a_{m-1}/2)}\left| \mathcal{S}_{\boldsymbol{a,v};(m,M)}\left(\{\Phi_{t}\}_{t>0}, f_{2}\right)(z)\right|dz\\
&\quad =:CA_{21}+CA_{22}.
\end{align*}

For $A_{21}$, notice that
\begin{align*}
&\left|\mathcal{S}_{\boldsymbol{a,v};(m,M)}\left(\{\Phi_{t}\}_{t>0}, f_{2}\right)(z)- \mathcal{S}_{\boldsymbol{a,v};(m,M)}\left(\{\Phi_{t}\}_{t>0}, f_{2}\right)(x_{0})\right|\\
&\quad \leq \int_{(B_{m})^{c}}\left| K_{\boldsymbol{a,v};(m,M)}(z-y)-K_{\boldsymbol{a,v};(m,M)}(x_{0}-y)\right| \left|f(y) \right|dy \\
&\quad=\sum\limits_{i=1}^{\infty}	 \int_{2^{i}B_{m}\setminus2^{i-1}B_{m}}\left| K_{\boldsymbol{a,v};(m,M)}(z-y)-K_{\boldsymbol{a,v};(m,M)}(x_{0}-y)\right| \left|f(y) \right|dy.
\end{align*}
For $z\in B(x_{0},a_{m-1}/2), y\in 2^{i}B_{m}\setminus2^{i-1}B_{m}, i\in \mathbb{N}$, by Lemma \ref{ S kernel es } (ii) and the mean value theorem, we know that there exists $\xi_{i}$ on the segment $\overline{x_{0}z}$ such that
\begin{align*}
&\left| K_{\boldsymbol{a,v};(m,M)}(z-y)-K_{\boldsymbol{a,v};(m,M)}(x_{0}-y)\right|\\
&\quad  =\left|\nabla K_{\boldsymbol{a,v};(m,M)}(\xi_{i}-y)\right|\left| x_{0}-z\right|\leq C\frac{\left| x_{0}-z\right|}{\left| \xi_{i}-y\right|^{n+1}}\leq C\frac{a_{m-1}}{\left( 2^{i}a_{m}\right) ^{n+1}}.
\end{align*}
Hence, for any $z\in B(x_{0},a_{m-1}/2)$,
\begin{align*}
	 &\left|\mathcal{S}_{\boldsymbol{a,v};(m,M)}\left(\{\Phi_{t}\}_{t>0}, f_{2}\right)(z)- \mathcal{S}_{\boldsymbol{a,v};(m,M)}\left(\{\Phi_{t}\}_{t>0}, f_{2}\right)(x_{0})\right|\\
	&\quad \leq C \sum\limits_{i=1}^{\infty}	 \int_{2^{i}B_{m}\setminus2^{i-1}B_{m}}\frac{a_{m-1}}{\left( 2^{i}a_{m}\right) ^{n+1}} \left|f(y) \right|dy\\
	&\quad \leq C \left( \sum\limits_{i=1}^{\infty}\frac1{2^i}\right) \mathcal{M}f(x_{0})\leq C\mathcal{M}_{q}f(x_{0}).
\end{align*}
It yields
$$A_{21}\leq C\mathcal{M}_{q}f(x_{0}).$$

For $A_{22}$,
\begin{align*}
A_{22}&=\frac{1}{\left(a_{m-1}/2\right)^{n} }\int_{B(x_{0},a_{m-1}/2)}\left| \mathcal{S}_{\boldsymbol{a,v};(-M,M)}\left(\{\Phi_{t}\}_{t>0}, f_{2}\right)(z)\right.\\
&\left.\quad \quad \quad\quad \quad \quad\quad \quad \quad\quad \quad \quad-\mathcal{S}_{\boldsymbol{a,v};(-M,m-1)}\left(\{\Phi_{t}\}_{t>0}, f_{2}\right)(z)\right|dz \\
&\leq \frac{1}{\left(a_{m-1}/2\right)^{n} }\int_{B(x_{0},a_{m-1}/2)}\left| \mathcal{S}_{\boldsymbol{a,v};(-M,M)}\left(\{\Phi_{t}\}_{t>0}, f\right)(z)\right| dz\\
&\quad +\frac{1}{\left(a_{m-1}/2\right)^{n} }\int_{B(x_{0},a_{m-1}/2)}\left| \mathcal{S}_{\boldsymbol{a,v};(-M,M)}\left(\{\Phi_{t}\}_{t>0}, f_{1}\right)\right| (z)dz\\
&\quad +\frac{1}{\left(a_{m-1}/2\right)^{n} }\int_{B(x_{0},a_{m-1}/2)}\left| \mathcal{S}_{\boldsymbol{a,v};(-M,m-1)}\left(\{\Phi_{t}\}_{t>0}, f_{2}\right)(z)\right| dz\\
&=: A_{221}+A_{222}+A_{223}.
\end{align*}
It is obvious that
$$A_{221}\leq \mathcal{M}\left(\mathcal{S}_{\boldsymbol{a,v};(-M,M)}\left(\{\Phi_{t}\}_{t>0}, f\right) \right) (x_{0}).$$

For $A_{222}$, by the H\"older inequality  and the boundedness of $\mathcal{S}_{\boldsymbol{a,v};(-M,M)}\left(\{\Phi_{t}\}_{t>0},\cdot\right)$,

\begin{align*}
A_{222}&\leq C\left[ \frac{1}{\left(a_{m-1}/2\right)^{n} }\int_{B(x_{0},a_{m-1}/2)}\left| \mathcal{S}_{\boldsymbol{a,v};(-M,M)}\left(\{\Phi_{t}\}_{t>0}, f_{1}\right) (z)\right|^{q}dz\right] ^{\frac1{q}}\\
&\leq C\left( \frac{1}{\left(a_{m-1}/2\right)^{n} }\int_{\rn}\left|  f_{1}(z)\right|^{q}dz\right)  ^{\frac1{q}}\\
&\leq C\left( \frac{1}{\left(a_{m}/2\right)^{n} }\int_{B_{m}}\left|  f(z)\right|^{q}dz\right)  ^{\frac1{q}}\leq C\mathcal{M}_{q}f(x_{0}).
\end{align*}

For $A_{223}$, notice that
$$A_{223}\leq C\frac{1}{\left(a_{m-1}/2\right)^{n} }\int_{B(x_{0},a_{m-1}/2)}\int_{\left( B_{m}\right)^{c} }\left| K_{\boldsymbol{a,v};(-M,m-1)}(z-y)\right|\left| f(y)\right|dy   dz.$$
And for any $z\in B(x_{0}, a_{m-1}/2)$, when $y\in B_{k+1}\setminus B_{k}$ with $k\geq m$, $\left|z-y \right|\sim \left| y-x_{0}\right| \geq a_{k} $. Then by Lemma \ref{for C ineq} (ii),
\begin{align*}
&\int_{\left( B_{m}\right)^{c} }\left| K_{\boldsymbol{a,v};(-M,m-1)}(z-y)\right|\left| f(y)\right|dy\\
&\quad = \sum\limits_{k=m}^{\infty} \int_{B_{k+1}\setminus B_{k}}\left|\sum\limits_{i=-M}^{m-1} v_{i}\left[ \phi_{a_{i+1}}(z-y)-\phi_{a_{i}}(z-y)\right] f(y)\right|dy\\
&\quad \leq C \sum\limits_{k=m}^{\infty} \int_{B_{k+1}\setminus B_{k}}\frac{\delta^{m-k-1}}{a_{k} ^n }\left|f(y)\right|dy\\
&\quad \leq C\sum\limits_{k=m}^{\infty} \int_{B_{k+1}}\frac{\delta^{m-k-1}}{a_{k+1}^n }\left|f(y)\right|dy\\
&\quad \leq C\sum\limits_{k=m}^{\infty}\delta^{m-k-1}\mathcal{M}f(x_{0})\leq C\mathcal{M}_{q}f(x_{0}).
\end{align*}
Hence,
$$A_{223}\leq C\mathcal{M}_{q}f(x_{0}).$$
Combining  the estimates above for $A_{221}, A_{222}$ and $A_{223}$, we get
$$A_{22}\leq C\mathcal{M}_{q}f(x_{0})+C\mathcal{M}\left(\mathcal{S}_{\boldsymbol{a,v};(-M,M)}\left(\{\Phi_{t}\}_{t>0}, f\right) \right) (x_{0}).$$
And the estimates for $A_{21}$ and $A_{22}$ imply
$$A_{2}\leq C\mathcal{M}_{q}f(x_{0})+C\mathcal{M}\left(\mathcal{S}_{\boldsymbol{a,v};(-M,M)}\left(\{\Phi_{t}\}_{t>0}, f\right) \right) (x_{0}).$$
Summing the estimates of $A_{1}$ and $A_{2}$, we conclude that
$$\mathcal{S}_{\boldsymbol{a,v};(m,M)}\left(\{\Phi_{t}\}_{t>0}, f\right)(x_0)\leq C\left\lbrace\mathcal{M}\left( \mathcal{S}_{\boldsymbol{a,v};(-M,M)}\left(\{\Phi_{t}\}_{t>0}, f\right)\right)(x_0)+\mathcal{M}_{q}f(x_0) \right\rbrace, $$
where the constant $C$ appearing above  depends on $n, p, \phi$ and $\left\| \boldsymbol{v}\right\|_{l^{\infty}(\mathbb{Z})}$ (not on $M$).
 Consequently, we have
$$ \mathcal{S}_{\boldsymbol{a,v};M}^{\ast}\left(\{\Phi_{t}\}_{t>0}, f\right)(x_0)\leq C\left\lbrace\mathcal{M}\left( \mathcal{S}_{\boldsymbol{a,v};(-M,M)}\left(\{\Phi_{t}\}_{t>0}, f\right)\right)(x_0)+\mathcal{M}_{q}f(x_0) \right\rbrace. $$
This completes the proof of Theorem \ref{C ineq}.
\end{proof}

Now we can give the proof of Theorem \ref{S Lp}.
\begin{proof}[\bf Proof of Theorem \ref{S Lp}]
For any $p\in(1,\infty)$, we can choose $1<q<p$. Then it is well known  that the maximal
operators $\mathcal{M}$ and $\mathcal{M}_{q}$ are both bounded on $L^{p}(\rn)$. Furthermore, by Theorem \ref{C ineq} and Proposition \ref{S partial Lp}, we have
\begin{align*}
	 &\|\mathcal{S}_{\boldsymbol{a,v};M}^{\ast}\left(\{\Phi_{t}\}_{t>0}, f\right)\|_{L^{p}(\rn)}\leq C\left( \| \mathcal{M}\left( \mathcal{S}_{\boldsymbol{a,v};(-M,M)}\left(\{\Phi_{t}\}_{t>0}, f\right)\right) \|_{L^{p}(\rn)}+\|\mathcal{M}_{q}f\|_{L^{p}(\rn)}\right)\\
	&\quad \leq C\left(\|  \mathcal{S}_{\boldsymbol{a,v};(-M,M)}\left(\{\Phi_{t}\}_{t>0}, f\right) \|_{L^{p}(\rn)}+\|f\|_{L^{p}(\rn)} \right)\leq C \|f\|_{L^{p}(\rn)}.
\end{align*}
Since the constant $C$ appearing above does not depend on $M$, letting $M$ increase to infinity, by Fatou's lemma, we have
$$\left\| \mathcal{S}_{\boldsymbol{a,v};*}\left(\{\Phi_{t}\}_{t>0}, f \right)\right\| _{L^p(\rn)}\leq C\|f\|
_{L^p(\rn)}.$$
This completes the proof of Theorem \ref{S Lp}.
\end{proof}

\subsection{ {\rm{BMO-BLO}} boundedness of the maximal differential transform}\label{subsec 4.2}
\begin{proof}[\bf Proof of Theorem \ref{main 3}]
Let $f\in\bmoz$. We first assume $\mathcal{S}_{\boldsymbol{a,v};*}\left(\{\Phi_{t}\}_{t>0}, f \right)(x)$ is finite almost everywhere. And in this case, we shall show that there exists a constant $C$  depending on $n, \left\| \boldsymbol{v}\right\|_{l^{\infty}(\mathbb{Z})}$, $\delta$ and $\phi$ such that for any ball $B:=B(x_{0},r)\subset \rn$.
\begin{align}
	F:=\frac{1}{\left| B\right| }\int_{B}&\left[ \mathcal{S}_{\boldsymbol{a,v};*}\left(\{\Phi_{t}\}_{t>0}, f\right)(x)\right.\notag\\
	&\left.-\mathop{\einf}\limits_{y\in B}\mathcal{S}_{\boldsymbol{a,v};*}\left(\{\Phi_{t}\}_{t>0}, f\right)(y)\right] dx\leq C  \|f\|_{\bmoz}.\label{main result S}
\end{align}

Since $\{a_{i}\}_{i\in \mathbb{Z}}$ is a $\delta$-lacunary sequence of positive number, there exists an integer $i_{0}$ such that $a_{i_{0}}<r\leq a_{i_{0}+1}$. We decompose $B$ into the following  three subsets:
\begin{align*}
B_{1}&:=\left\lbrace x\in B:\mathcal{S}_{\boldsymbol{a,v};*}\left(\{\Phi_{t}\}_{t>0}, f\right)(x)=\sup_{\substack{N_{2}\leq i_{0}-1\\(N_{1},N_{2})\in{\mathbb{Z}_{<}^{2}}} } \left| \sum\limits_{i=N_{1}}^{N_{2}}v_{i}\left(\Phi_{a_{i+1}}f(x)-\Phi_{a_{i}}f(x) \right) \right|\right\rbrace,\\
B_{2}&:=\left\lbrace x\in B:\mathcal{S}_{\boldsymbol{a,v};*}\left(\{\Phi_{t}\}_{t>0}, f\right)(x)=\sup_{\substack{N_{1}\geq i_{0}+1\\(N_{1},N_{2})\in{\mathbb{Z}_{<}^{2}} }} \left| \sum\limits_{i=N_{1}}^{N_{2}}v_{i}\left(\Phi_{a_{i+1}}f(x)-\Phi_{a_{i}}f(x) \right) \right|\right\rbrace,
\end{align*}
{and}\begin{equation*}
B_{3}:=\left\lbrace x\in B:\mathcal{S}_{\boldsymbol{a,v};*}\left(\{\Phi_{t}\}_{t>0}, f\right)(x)=\sup_{\substack{N_{1}\leq i_{0}\leq N_{2}\\(N_{1},N_{2})\in{\mathbb{Z}_{<}^{2}} }} \left| \sum\limits_{i=N_{1}}^{N_{2}}v_{i}\left(\Phi_{a_{i+1}}f(x)-\Phi_{a_{i}}f(x) \right) \right|\right\rbrace.
\end{equation*}
For each $j=1, 2, 3$, we define $F_{j}$ as $F$ with $B$ replaced by $B_{j}$ respectively. Then for $F_{1}$, we have
\begin{align*}
F_{1}\leq\frac1{\left|B\right| }\int_{B_{1}}\sup_{\substack{N_{2}\leq i_{0}-1\\ (N_{1},N_{2})\in{\mathbb{Z}_{<}^{2}} }} \left| \sum\limits_{i=N_{1}}^{N_{2}}v_{i}\left(\Phi_{a_{i+1}}f(x)-\Phi_{a_{i}}f(x) \right) \right|dx.
\end{align*}
Since $a_{i_{0}-1}<a_{i_{0}}<r$, by proceeding as in the estimations of $G_{1}, G_{2}, G_{3}$ in the proof of Theorem \ref{main O}, we can see
$$F_{1}\leq C \|f\|_{\bmoz}. $$

For $F_{2}$, we have
\begin{align*}
F_{2}&\leq \frac1{\left|B \right| }\int_{B_{2}}\left\lbrace \sup_{\substack{N_{1}\geq i_{0}+1\\(N_{1},N_{2})\in{\mathbb{Z}_{<}^{2}} }} \left| \sum\limits_{i=N_{1}}^{N_{2}}v_{i}\left(\Phi_{a_{i+1}}f(x)-\Phi_{a_{i}}f(x) \right) \right|\right.\\
&\left.\quad-\mathop{\einf}\limits_{y\in B}\sup_{\substack{N_{1}\geq i_{0}+1\\(N_{1},N_{2})\in{\mathbb{Z}_{<}^{2}} }} \left| \sum\limits_{i=N_{1}}^{N_{2}}v_{i}\left(\Phi_{a_{i+1}}f(y)-\Phi_{a_{i}}f(y) \right) \right| \right\rbrace dx\\
&\leq \mathop{\esup}\limits_{x,y\in B}\sup_{\substack{N_{1}\geq i_{0}+1\\(N_{1},N_{2})\in{\mathbb{Z}_{<}^{2}} }}\left| \sum\limits_{i=N_{1}}^{N_{2}}v_{i}\left[ \left(\Phi_{a_{i+1}}f(x)-\Phi_{a_{i}}f(x) \right)-\left(\Phi_{a_{i+1}}f(y)-\Phi_{a_{i}}f(y) \right)\right] \right|.
\end{align*}
Since $a_{i_{0}+1}\geq r$, by proceeding as in the estimations of $E_{2}$ in the proof of Theorem \ref{main O}, we can see
$$F_{2}\leq C \|f\|_{\bmoz}. $$

For $F_{3}$, we can see,  for any $x\in B_{3}$,
\begin{align*}
&\sup_{\substack{N_{1}\leq i_{0}\leq N_{2}\\(N_{1},N_{2})\in{\mathbb{Z}_{<}^{2}}} } \left| \sum\limits_{i=N_{1}}^{N_{2}}v_{i}\left(\Phi_{a_{i+1}}f(x)-\Phi_{a_{i}}f(x) \right) \right|\\
&\quad \leq \sup_{N_{1}<i_{0}} \left| \sum\limits_{i=N_{1}}^{i_{0}}v_{i}\left(\Phi_{a_{i+1}}f(x)-\Phi_{a_{i}}f(x) \right) \right|+\sup_{N_{1}\leq i_{0}<i_{0}+1\leq N_{2} } \left| \sum\limits_{i=N_{1}}^{N_{2}}v_{i}\left(\Phi_{a_{i+1}}f(x)-\Phi_{a_{i}}f(x) \right) \right|\\
&\quad \leq \sup_{N_{1}<i_{0}} \left| \sum\limits_{i=N_{1}}^{i_{0}+1}v_{i}\left(\Phi_{a_{i+1}}f(x)-\Phi_{a_{i}}f(x) \right) \right|+\left|v_{i_{0}+1}\left( \Phi_{a_{i_{0}+2}}f(x)-\Phi_{a_{i_{0}+1}}f(x) \right) \right| \\
&\quad \quad +\sup_{N_{1}\leq i_{0}}\left| \sum\limits_{i=N_{1}}^{i_{0}+1}v_{i}\left(\Phi_{a_{i+1}}f(x)-\Phi_{a_{i}}f(x) \right) \right|+\sup_{N_{2}\geq i_{0}+1}\left| \sum\limits_{i=i_{0}+1}^{N_{2}}v_{i}\left(\Phi_{a_{i+1}}f(x)-\Phi_{a_{i}}f(x) \right) \right|\\
&\quad \leq 3\sup_{N_{1}\leq i_{0}+1}\left| \sum\limits_{i=N_{1}}^{i_{0}+1}v_{i}\left(\Phi_{a_{i+1}}f(x)-\Phi_{a_{i}}f(x) \right) \right|+\sup_{N_{2}\geq i_{0}+1}\left| \sum\limits_{i=i_{0}+1}^{N_{2}}v_{i}\left(\Phi_{a_{i+1}}f(x)-\Phi_{a_{i}}f(x) \right) \right|.
\end{align*}
It follows that,
\begin{align*}
F_{3}&\leq \frac{3}{\left|B \right| }\int_{B}\sup_{N_{1}\leq i_{0}+1}\left| \sum\limits_{i=N_{1}}^{i_{0}+1}v_{i}\left(\Phi_{a_{i+1}}f(x)-\Phi_{a_{i}}f(x) \right) \right|dx\\	
&\quad +\mathop{\esup}\limits_{x,y\in B}	\sup_{N_{2}\geq i_{0}+1 }\left| \sum\limits_{i=i_{0}+1}^{N_{2}}v_{i}\left[ \left(\Phi_{a_{i+1}}f(x)-\Phi_{a_{i}}f(x) \right)-\left(\Phi_{a_{i+1}}f(y)-\Phi_{a_{i}}f(y) \right)\right] \right|\\
&=:3F_{31}+F_{32}.
\end{align*}

From $\eqref{add}$ and $a_{i_{0}}<r$, we deduce $a_{i_{0}+2}\leq \delta^{4}a_{i_{0}} \leq Cr$. Then the arguments of $F_{31}$ and $F_{32}$ are similar to $F_{1}$ and $F_{2}$ respectively. And then we get
$$F_{3}\leq  C \|f\|_{\bmoz}. $$

 Combining the estimates above for $F_{1}, F_{2}$ and $F_{3}$, we prove \eqref{main result S}.

Now we only need to show that, if there exists a point $x_{0}\in\rn$ such that
\begin{equation}\label{Sae}
	\mathcal{S}_{\boldsymbol{a,v};*}\left(\{\Phi_{t}\}_{t>0}, f\right)(x_0)<\infty,
\end{equation}
then $\mathcal{S}_{\boldsymbol{a,v};*}\left(\{\Phi_{t}\}_{t>0}, f\right)$ is finite almost everywhere. Similar to the proof of \eqref{main result S}, we can prove that \begin{equation}\label{S a.e }
	\frac{1}{\left| B\right| }\int_{B}\left[ \mathcal{S}_{\boldsymbol{a,v};*}\left(\{\Phi_{t}\}_{t>0}, f\right)(x)-\inf\limits_{y\in B}\mathcal{S}_{\boldsymbol{a,v};*}\left(\{\Phi_{t}\}_{t>0}, f\right)(y)\right] dx\leq C  \|f\|_{\bmoz}.
\end{equation}
 By the assumption \eqref{Sae}, we can see that $\inf_{y\in B} \mathcal{S}_{\boldsymbol{a,v};*}\left(\{\Phi_{t}\}_{t>0}, f\right)(y) <\infty.$ Then \eqref{S a.e } implies  $$\mathcal{S}_{\boldsymbol{a,v};*}\left(\{\Phi_{t}\}_{t>0}, f\right)(x)<\infty,a.e.\ x\in B(x_0, r).$$
By the arbitrariness of $r$, we get
$$\mathcal{S}_{\boldsymbol{a,v};*}\left(\{\Phi_{t}\}_{t>0}, f\right)(x)<\infty,a.e.\ x\in \rn.$$
This completes the proof of Theorem \ref{main 3}.
\end{proof}

{\bf Acknowledgments:}{ The authors would like to thank the referees for careful reading
	and helpful suggestions, which help to make this paper more readable. This work is supported by the National Natural Science Foundation of China (Grant Nos. 11971402, 11971431, 12171399) and the Natural Science Foundation of Zhejiang Province (Grant No. LY22A010011). }

\vspace{0.3cm}

College of Mathematics and Information Science,\ Henan Normal University,\  Xinxiang 453007,\  P.R. China

\smallskip

{\it E-mail}: \texttt{huwenting@htu.edu.cn}

\vspace{0.3cm}

School of Mathematical Sciences,\  Xiamen University,\  Xiamen 361005,\ P.R. China

\smallskip

{\it E-mail}: \texttt{19020211153414@stu.xmu.edu.cn }

\vspace{0.3cm}

School of Mathematical Sciences,\ Xiamen University,\ Xiamen 361005,\  P.R. China

\smallskip

{\it E-mail}: \texttt{dyyang@xmu.edu.cn }

\vspace{0.3cm}

School of Statistics and Mathematics,\ Zhejiang Gongshang  University,\ Hangzhou 310018, \quad P.R. China

\smallskip

{\it E-mail}: \texttt{zaoyangzhangchao@163.com }

\end{document}